\documentclass[12pt,a4paper]{article}

\usepackage{color}
\usepackage{amsmath,amsfonts,amsthm,amssymb}
\usepackage{mathrsfs}
\usepackage{parskip}
\usepackage[T1]{fontenc}
\usepackage[utf8]{inputenc}
\usepackage{authblk}
\usepackage{enumerate}
\usepackage{graphicx}

\newtheorem{theorem}{Theorem}
\newtheorem{lemma}{Lemma}

\let\scr\mathscr

\def \Liminf{\mathop{\underline{\lim}}\limits}

\def\AA{\mathbb{A}}

\def\Pb{\mathbf{P}}

\def\Ex{\mathbf{E}}

\def\NN{\mathbb{N}}

\def\UU{\mathbb{U}}
\def\QQ{\mathbb{Q}}

\def\1{\mbox{1\hspace{-.25em}I}}
\begin{document}
\title{Hidden Markov Model Where Higher Noise Makes Smaller Errors}

\author[]{Yu.A. Kutoyants}
\affil[]{\small Le Mans University,  Le Mans,  France,\\
Tomsk State University, Tomsk, Russia}

\date{}

\maketitle
\begin{abstract}
We consider the problem of parameter estimation in a partially observed
linear Gaussian system  with small noises in the state and
observation equations. We describe asymptotic properties of the MLE and
Bayes estimators in the setting with state and observation noises of
possibly unequal intensities. It is shown that both estimators are
consistent, asymptotically normal with convergent moments and asymptotically
efficient. This model has an unusual feature: larger noise in the state equation
yields smaller estimation error. The proofs are based on asymptotic analysis of
the Kalman-Bucy filter and the associated Riccati equation in particular.
\end{abstract}
\noindent MSC 2000 Classification: 62M02,  62G10, 62G20.

\bigskip
\noindent {\sl Key words}: \textsl{Partially observed linear system, parameter
  estimation,  small noise asymptotic, asymptotic properties.}

\section{Introduction}

We consider partially observed stochastic linear system
\begin{align}
\label{01}
{\rm d}X_t&=f\left( \vartheta, t\right)Y_t{\rm
   d}t+\varepsilon \sigma\left(t\right){\rm d}W_t ,\qquad X_0=0, \qquad 0\leq
t\leq T ,\\
{\rm d}Y_t&=a\left(\vartheta ,t\right)Y_t{\rm d}t+\psi_\varepsilon
b\left(\vartheta ,t\right) {\rm d}V_t,\quad Y_0=y_0,\quad 0\leq
t\leq T ,
\label{02}
\end{align}
where $W_t, 0\leq t\leq T$ and $V_t,0\leq t\leq T$ are two independent Wiener
processes, $f\left(\cdot ,\cdot \right),a\left(\cdot ,\cdot
\right),b\left(\cdot ,\cdot \right)$ and $\sigma \left(\cdot \right)$ are
known functions, $\varepsilon\in (0,1] $ and $\psi_\varepsilon \in (0,1]$
    are noise intensities. The initial value $y_0$ is deterministic.
    The parameter $\vartheta \in\Theta =\left(\alpha ,\beta \right)$ is
    unknown and has to be estimated using the observations $X^T=\left(X_t,0\leq t\leq
    T\right)$. The Gaussian process $Y^T=\left(Y_t,0\leq t\leq T\right)$ is
    unobservable (hidden).
Construction of the maximum likelihood estimator (MLE) $\hat\vartheta
_\varepsilon $ and the Bayesian estimator (BE) $\tilde\vartheta _\varepsilon $ is
based on the likelihood  function \cite{LS01}
\begin{align}
\label{03}
L\left(\vartheta ,X^T\right)=\exp\left\{\int_{0}^{T}\frac{M\left(\vartheta
  ,t\right)}{\varepsilon ^2\sigma \left(t\right)^2}{\rm
  d}X_t-\int_{0}^{T}\frac{M\left(\vartheta ,t\right)^2}{2\varepsilon ^2\sigma
  \left(t\right)^2}{\rm d}t\right\},\qquad \vartheta \in\Theta .
\end{align}
Here $M\left(\vartheta ,t\right)=f\left(\vartheta ,t\right)m\left(\vartheta
,t\right)$ and the conditional expectation $m\left(\vartheta
,t\right)=\Ex_{\vartheta }\left(Y_t|X_s,0\leq s\leq t\right)$ is the solution of
the Kalman-Bucy filtering equations \cite{KB61}.

The MLE is solution of the following equation
\begin{align}
\label{04}
L\left(\hat\vartheta
_\varepsilon ,X^T\right)=\sup_{\vartheta \in\Theta }L\left(\vartheta ,X^T\right).
\end{align}
If this equation has more than one solution, then any one of them can be taken
as MLE.

To introduce BE we assume that the unknown parameter $\vartheta $ is a random variable
with  known density $p\left(\vartheta \right),\vartheta \in\Theta $. Then for
quadratic loss function the BE is the conditional expectation
\begin{align}
\label{05}
\tilde\vartheta _\varepsilon =\int_{\Theta }^{}\vartheta p\left(\vartheta
|X^T\right)\,{\rm d}\vartheta  ,\qquad p\left(\vartheta
|X^T\right)=\frac{p\left(\vartheta \right)L\left(\vartheta
  ,X^T\right)}{\int_{\Theta }^{}p\left(\theta \right)L\left(\theta
,X^T\right) {\rm d}\theta }.
\end{align}

Systems as \eqref{01} and \eqref{02}, either in continuous or discrete time, are
the typical models to which the Kalman-Bucy method is applicable.
Nowadays it is widely used in various branches of sciences and technology: industrial production \cite{AHG13}, GPS localization
\cite{HCCL03},\cite{Kut20b},\cite{ZB11}, chemistry and biochemistry
\cite{DD03},\cite{DM92}, \cite{Ru91}, physics \cite{KSGT07}, finance
\cite{W10}. This is only a short list, which can be easily extended. Note that
similar models   were studied in non linear filtration with small
noise in observations (see  \cite{FP89},\cite{Pic91} and references there in).

The engineering literature on adaptive Kalman-Bucy filtering is vast, but mathematical
study of statistical problems for such systems is not yet sufficiently
developed. Statistical problems for discrete time models were studied more extensively, see e.g.
monographs \cite{CMT05}, \cite{EAM95} and \cite{EM02}.  The continuous time
hidden Markov processes with discrete state space and  white Gaussian
noise observations were studied in \cite{PCh09},\cite{KhK18}.

For continuous time linear  systems such as \eqref{01}-\eqref{02} with
constant functions  $f\left(\vartheta ,t \right)=f\left(\vartheta
\right),a\left(\vartheta ,t \right)=a\left(\vartheta \right),b\left(\vartheta
,t \right)=b\left(\vartheta \right),\sigma \left(t \right)=\sigma $,
($\varepsilon =1,\psi _\varepsilon =1$), asymptotic analysis with respect to $T\rightarrow \infty$
appeared in  \cite{Kut84}, \cite{KS91}, \cite{Kut19b}. The survey \cite{Kut20a} reviews some results
on parameter estimation in the model \eqref{01}-\eqref{02} in both large time and small noise asymptotics.

The partially observed system  \eqref{01}-\eqref{02} with $\psi
 _\varepsilon =\varepsilon \rightarrow 0$ as well as some of its generalizations were
 studied in \cite{Kut94}, Chapter 6. Let us briefly recall some of the results in \cite{Kut94} and
 compare them with with those  obtained recently. The processes $X^T,Y^T$ as $\varepsilon \rightarrow 0$
 converge to the deterministic solutions of the ordinary differential
 equations
\begin{align}\label{06}
\frac{{\partial }x_t\left(\vartheta \right)}{{\partial}t}=f\left(\vartheta
,t\right)y_t\left(\vartheta \right), \quad \qquad
\frac{{\partial}y_t\left(\vartheta \right)}{{\partial}t}=-a\left(\vartheta
,t\right)y_t\left(\vartheta \right),
\end{align}
with initial values $x_0\left(\vartheta \right)=0 $ and $ y_0\left(\vartheta
\right)=y_0$ respectively. It is shown that under appropriate regularity conditions the
MLE $\hat\vartheta _\varepsilon $ and BE $\tilde\vartheta _\varepsilon $ are
consistent, asymptotically normal
\begin{align}
\label{07}
\frac{\hat\vartheta _\varepsilon -\vartheta _0}{\varepsilon }\Longrightarrow
\hat\zeta \sim {\cal N}\left(0,\hat{\rm I}\left(\vartheta _0\right)^{-1}\right),\qquad
\quad \frac{\tilde\vartheta _\varepsilon -\vartheta _0}{\varepsilon
}\Longrightarrow \hat\zeta,
\end{align}
the moments converge and  both estimators are asymptotically efficient. Here
$\vartheta _0$ is the true parameter value and $\hat{\rm I}\left(\vartheta _0\right) $ is
the Fisher information,
\begin{align*}
&\int_{0}^{T}\frac{\dot  M\left(\vartheta_0,t \right)^2}{\sigma\left(t\right) ^{2}}{\rm d}t \longrightarrow \int_{0}^{T}\frac{\left[\dot f\left(\vartheta_0
  ,t\right)y_t\left(\vartheta_0 \right)+f\left(\vartheta_0 ,t\right)\dot
  y\left(\vartheta_0,t \right)\right]^2}{\sigma\left(t\right) ^{2}}{\rm
  d}t=\hat{\rm I}\left(\vartheta_0\right).
\end{align*}
Here and below the derivative in $\vartheta$ are denoted by dots and
$$
\dot
  M\left(\vartheta_0,t \right)=\dot f\left(\vartheta_0
  ,t\right)m\left(\vartheta_0 ,t\right)+f\left(\vartheta_0 ,t\right)\dot
  m\left(\vartheta_0,t \right).
$$
The function $\dot y\left(\vartheta_0,t \right)$ is distinct from $\dot y_t\left(\vartheta \right)$ and is the solution of a
certain auxiliary linear equation. It is important to note  that if $y_0=0$, then $\hat{\rm I}\left(\vartheta\right)=0
  $. Therefore $y_0\not=0$ is a necessary condition for  \eqref{07}.

Another instance of the model \eqref{01}-\eqref{02} with $\psi _\varepsilon =1$ was studied in
\cite{Kut19a}. In this case the limit system is
\begin{align}
\label{08}
\tilde x_t&=f\left(\vartheta
,t\right)Y_t\left(\vartheta \right),\qquad \quad \tilde x_0=0, \\
{\rm d}Y_t&=a\left(\vartheta ,t\right)Y_t{\rm d}t+
b\left(\vartheta ,t\right) {\rm d}V_t,\quad Y_0=0,\quad 0\leq
t\leq T ,
\label{09}
\end{align}
where we denoted $\tilde x_t=\frac{{\partial }X_t\left(\vartheta
  \right)}{{\partial}t} $.  Roughly speaking the question of consistency of
estimators is reduced to the following one: {\it is it possible to determine
  the parameter $\vartheta $ exactly, using the observations $\tilde x^T=(\tilde x_t,0\leq
  t\leq T)$}? In \cite{Kut19a} it was shown that if the identifiability
condition
\begin{align}
\label{10}
\inf_{\left|\vartheta -\vartheta _0\right|>\nu
}\int_{0}^{T}\frac{\left[S\left(\vartheta ,t\right)-S\left(\vartheta_0
  ,t\right)\right]^2}{S\left(\vartheta ,t\right)\sigma \left(t\right)}{\rm d}t >0,\quad \forall \nu>0,
\end{align}
where $S\left(\vartheta ,t\right)=f\left(\vartheta ,t\right)b\left(\vartheta
,t\right) $, is satisfied, then the answer to the above question is in positive.

Moreover, under mild regularity conditions the MLE and BE are
asymptotically normal
\begin{align}
\label{11}
\frac{\hat\vartheta _\varepsilon -\vartheta _0}{\sqrt{\varepsilon} }\Longrightarrow
\zeta \sim {\cal N}\left(0,{\rm I}\left(\vartheta _0\right)^{-1}\right),\qquad
\quad \frac{\tilde\vartheta _\varepsilon -\vartheta _0}{\sqrt{\varepsilon}
}\Longrightarrow \zeta,
\end{align}
 the moments converge and these estimators are asymptotically efficient. Here
 $\hat{\rm I}\left(\vartheta\right) $ the Fisher information is given by a different expression.

This model has an interesting feature. If $f\left(\vartheta
 ,t\right)=f\left(t\right), b\left(\vartheta ,t\right)=b\left(t\right)$ or
 $f\left(\vartheta ,t\right)=\vartheta f\left(t\right), b\left(\vartheta
 ,t\right)=\vartheta ^{-1}b\left(t\right)$ then the
 identifiability condition \eqref{10} fails and  convergence \eqref{11}
does not hold. Note also that condition \eqref{10} does not depend on
 $a\left(\vartheta ,t\right)$ and $y_0=0$. Remark that the condition $y_0=0$
was omitted in \cite{Kut19a}, but the limit of estimators in the case
$y_0\not=0$ is different. 

It was shown that
\begin{align}
\label{12}
\int_{0}^{T}\frac{\dot  M\left(\vartheta_0,t \right)^2}{\sigma\left(t\right) ^{2}}\;{\rm
  d}t\xrightarrow{\varepsilon  \rightarrow 0}  0.
\end{align}
and
\begin{align*}
\frac{\dot f\left(\vartheta_0
  ,t\right)m\left(\vartheta_0 ,t\right)+f\left(\vartheta_0 ,t\right)\dot
  m\left(\vartheta_0,t \right)}{\sqrt{\varepsilon }\sigma \left(t\right)}\Longrightarrow h\left(t\right)\xi
_t,\qquad t\in (0,T],
\end{align*}
where $h\left(\cdot \right)$ is a bounded function and $\xi _t,t\in (0,T]$ is
  a family of independent Gaussian ${\cal N}\left(0,1\right)$ random
  variables. However, the convergence
\begin{align*}
\int_{0}^{T}\frac{\dot  M\left(\vartheta_0,t \right)^2}{\varepsilon \sigma\left(t\right) ^{2}}\;{\rm
  d}t\Longrightarrow \int_{0}^{T}h\left(t\right)^2\xi
_t^2{\rm d}t,
\end{align*}
does not hold for a number of reasons. In particular, the latter integral does not exist in any reasonable
sense because $\xi_t,t\in (0,T]$ is not a separable process. Instead, it is shown in \cite{Kut19a} that
\begin{align}
\label{13}
&\int_{0}^{T}\frac{\dot  M\left(\vartheta_0,t \right)^2}{\varepsilon \sigma\left(t\right) ^{2}}\;{\rm
  d}t\longrightarrow \int_{0}^{T}h\left(t\right)^2{\rm
  d}t =\int_{0}^{T}\frac{\dot S\left(\vartheta_0
  ,t\right)^2}{2S\left(\vartheta_0 ,t\right)\sigma \left(t\right)}{\rm d}t ={\rm I}\left(\vartheta\right).
\end{align}

The present work is concerned the setting, intermediate between these two. We
consider the model \eqref{01}-\eqref{02} with
$\varepsilon \rightarrow 0$, $\psi _\varepsilon \rightarrow 0$ and
\begin{align}
\label{14}
\frac{\varepsilon }{\psi _\varepsilon ^3}\;\longrightarrow \;0.
\end{align}
For $\psi _\varepsilon =\varepsilon ^\delta $,  \eqref{14}
corresponds to $\delta \in \left(0,\frac{1}{3}\right)$. The limit system in
this case coincides with \eqref{06}, we have the convergence  \eqref{12} and
the convergence as in  \eqref{13} but with different normalization
\begin{align*}
\int_{0}^{T}\frac{\dot  M\left(\vartheta_0,t \right)^2}{\varepsilon\psi _\varepsilon \;\sigma\left(t\right)
  ^{2}}\;{\rm d}t\longrightarrow {\rm I}\left(\vartheta_0\right)
\end{align*}
where  Fisher information is given by \eqref{13}.
It will be shown that under suitable regularity conditions, the MLE and BE are consistent,
asymptotically normal
\begin{align}
\label{conv}
\sqrt{\frac{\psi _\varepsilon }{\varepsilon }}\left(\hat\vartheta _\varepsilon
-\vartheta _0\right)\Longrightarrow
\zeta \sim {\cal N}\left(0,{\rm I}\left(\vartheta _0\right)^{-1}\right),\qquad
\sqrt{\frac{\psi _\varepsilon }{\varepsilon }}\left(\tilde\vartheta
_\varepsilon -\vartheta _0\right)\Longrightarrow \zeta,
\end{align}
the moments converge and the both estimators are asymptotically efficient.

The error asymptotics  implied by \eqref{conv}  is somewhat surprising,
\begin{align*}
\Ex_{\vartheta _0}\left(\hat\vartheta _\varepsilon -\vartheta _0
\right)^2=\frac{\varepsilon\left(1+o\left(1\right)\right) }{\psi _\varepsilon
  {\rm I}\left(\vartheta _0\right)} ,\qquad \Ex_{\vartheta
  _0}\left(\tilde\vartheta _\varepsilon -\vartheta _0
\right)^2=\frac{\varepsilon\left(1+o\left(1\right)\right) }{\psi _\varepsilon
  {\rm I}\left(\vartheta _0\right)}.
\end{align*}
This means that {\it larger noise in the state equation causes  smaller  estimation errors}
(Theorem \ref{T1} below).  The best case corresponds to the
situation, where the noise in the state equation does not tend to zero ($\psi
_\varepsilon =1$). We say {\it surprising} because for the values $\psi
_\varepsilon =\varepsilon ^\delta , \delta \in (\frac{1}{3},1]$ the situation changes
  essentially,  $\psi _\varepsilon $ became a  {\it true noise} and the
  normalization of estimators is different (see section 3).

Convergence of MLE and BE for the model \eqref{01}-\eqref{02} with $\psi _\varepsilon =\varepsilon $
has a different rate in the case of change-point in the observation equation, i.e., if $f\left(\vartheta
,t\right)=q\left(t\right)\1_{\left\{t<\vartheta
  \right\}}+r\left(t\right)\1_{\left\{t\geq \vartheta \right\}}$. Then it can be shown
that
\begin{align*}
\frac{\hat\vartheta _\varepsilon -\vartheta _0}{\varepsilon
  ^2}\Longrightarrow \zeta _*,\qquad \frac{\tilde\vartheta _\varepsilon
  -\vartheta _0}{\varepsilon ^2}\Longrightarrow \zeta^*,
\end{align*}
where $\zeta _*$ and $\zeta ^*$ are two different random
variables. Only the BE is asymptotically efficient in this case.
When the change-point is introduced into the state equation
$a\left(\vartheta,t\right)=q\left(t\right)\1_{\left\{t<\vartheta
  \right\}}+r\left(t\right)\1_{\left\{t\geq \vartheta \right\}} $, the
estimators are asymptotically normal with {\it regular rate}  $\varepsilon $
as in \eqref{07}, see \cite{Kut94}.

\section{Main result}

Consider the observation model
\begin{align*}
{\rm d}X_t&=f\left( \vartheta, t\right)Y_t\;{\rm
   d}t+\varepsilon \sigma\left(t\right)\;{\rm d}W_t ,\qquad X_0=0, \qquad 0\leq
t\leq T ,\\
{\rm d}Y_t&=a\left(\vartheta ,t\right)Y_t\;{\rm d}t+\psi_\varepsilon
b\left(\vartheta ,t\right) \;{\rm d}V_t,\quad Y_0=0,\quad 0\leq
t\leq T.
\end{align*}
Our goal is to estimate $\vartheta$ using the observations $X^T$. To this end, we will study the
asymptotic behavior of  the MLE and BE  defined in \eqref{04} and \eqref{05} respectively.
In the case of BE we assume that the density $p\left(\cdot \right)
 $ is a continuous positive function on $\Theta=\left(\alpha ,\beta \right)
 $.

 These estimators are based on the family of stochastic  processes
 $(m\left(\vartheta ,t\right),$ $0\leq t\leq T),$ $ \vartheta \in\Theta
 $,
where the conditional expectation  $m\left(\vartheta ,t\right),0\leq t\leq T$ satisfies the
 Kalman-Bucy filtering equations \cite{KB61} (see details in
\cite{LS01}, Theorem 10.1)
\begin{align}
\label{2-1}
{\rm d}m\left(\vartheta ,t\right)&=-a(\vartheta,t)m\left(\vartheta ,t\right){\rm d}t\nonumber\\
&\qquad \qquad +
\frac{\gamma \left(\vartheta ,t\right)f\left(\vartheta,
  t\right)}{\varepsilon ^2\sigma(t)^2} \left[{\rm d}X_t-f\left(\vartheta,
  t\right)m\left(\vartheta ,t\right){\rm d}t\right],
\end{align}
subject to $m\left(\vartheta ,0\right)=0$, where the function $\gamma \left(\vartheta
,t\right)=\Ex_\vartheta \left(m\left(\vartheta ,t\right)-Y_t \right)^2 $ solves
the Riccati equation
\begin{align}
\label{2-2}
\frac{\partial \gamma \left(\vartheta ,t\right)}{\partial t}=
-2a(\vartheta,t)\gamma \left(\vartheta ,t\right) -\frac{\gamma \left(\vartheta
  ,t\right)^2f\left(  \vartheta t\right)^2}{ \varepsilon ^2\sigma(t)^2  } +\psi_\varepsilon
^2b(\vartheta,t)^2,
\end{align}
subject to initial condition $\gamma \left(\vartheta ,0\right)=0  $.

The true value will be denoted by $\vartheta _0$. The equations \eqref{2-1} and
 \eqref{2-2} generate the conditional expectation if $\vartheta =\vartheta
 _0$. For other values of $\vartheta $ these equations define stochastic processes,
 which do not coincide in general with the conditional expectation  of $Y_t$.
 It will be convenient to introduce the function $\gamma_* \left(\vartheta
 ,t\right)= \gamma \left(\vartheta ,t\right)/\left(\varepsilon \psi
   _\varepsilon\right) $. Equations \eqref{2-1} and
 \eqref{2-2}  can be written as
\begin{align}
\label{2-3}
{\rm d}m\left(\vartheta ,t\right)&=-q_\varepsilon \left(\vartheta ,t\right)m\left(\vartheta
,t\right){\rm d}t +\frac{\psi_\varepsilon
}{\varepsilon } \frac{\gamma_* \left(\vartheta
  ,t\right)f\left(\vartheta, t\right)}{\sigma(t)^2}\; {\rm d}X_t,
\\
\frac{\varepsilon }{\psi_\varepsilon }\frac{\partial
  \gamma _*\left(\vartheta ,t\right)}{\partial t}&= -\frac{2\varepsilon
}{\psi_\varepsilon }a(\vartheta,t)\gamma _*\left(\vartheta
,t\right) -\frac{\gamma_* \left(\vartheta ,t\right)^2f\left(
  \vartheta, t\right)^2}{ \sigma(t)^2 }  +b(\vartheta,t)^2,
\label{2-4}
\end{align}
subject to $m\left(\vartheta ,0\right)=0,\gamma_* \left(\vartheta ,0\right)=0 $ and where
\begin{align*}
q_\varepsilon \left(\vartheta ,t\right)=a(\vartheta,t)+\frac{\psi_\varepsilon
}{\varepsilon }\frac{\gamma_* \left(\vartheta ,t\right)f\left(\vartheta,
  t\right)^2}{\sigma(t)^2}.
\end{align*}
 In the case  $\vartheta =\vartheta _0$ we have
\begin{align*}
{\rm d}m\left(\vartheta_0 ,t\right)&=-a(\vartheta_0,t)m\left(\vartheta_0
,t\right){\rm d}t +{\psi_\varepsilon } \frac{\gamma_* \left(\vartheta_0
  ,t\right)f\left(\vartheta_0, t\right)}{\sigma(t)}\; {\rm d}\bar W_t,\;
m\left(\vartheta _0,0\right)=0.
\end{align*}
 Here $\bar W_t,0\leq
t\leq T$ is the {\it innovation} Wiener process (see Theorem 7.12 in \cite{LS01}).

Finding the MLE $\hat\vartheta _\varepsilon $ and BE $\tilde\vartheta _\varepsilon $ by means of \eqref{04},\eqref{05}
requires solving equations \eqref{2-3},\eqref{2-4} for all $\vartheta \in\Theta$ and is therefore computationally inefficient.
In Section 3 below we discuss the possibility of estimating $\vartheta $ using a much more simple algorithm.

The properties of estimators will be derived under the following regularity conditions.

{\it Conditions} {${\cal A}$.}
\begin{description}
\item[${\cal A}_1$.]{\it The functions $f\left(\vartheta  ,t\right), a\left(\vartheta,
  t\right),b\left(\vartheta, t\right),  t\in \left[0,T\right],\vartheta
  \in\Theta$  and $\sigma \left(t\right), t\in \left[0,T\right]$ have
  continuous derivatives in $t$.   }
\item[${\cal A}_2$.]{\it The functions $f\left(\vartheta
  ,t\right),b\left(\vartheta
  ,t\right), t\in \left[0,T\right],\vartheta
  \in\Theta$  and  $\sigma \left(t\right),0\leq t\leq T$  are  separated from zero.   }
\item[${\cal A}_3$.]{\it The functions $f\left(\vartheta
  ,t\right), a\left(\vartheta,
  t\right),b\left(\vartheta, t\right), t\in \left[0,T\right],\vartheta
  \in\Theta$  are   two times continuously differentiable in $\vartheta  $
  and the derivatives $\dot f\left(\vartheta ,t\right),\dot b\left(\vartheta
  ,t\right)$ have continuous derivatives in $t$. }
\item[${\cal A}_4$.]{\it The function $\psi_\varepsilon=\varepsilon ^\delta \rightarrow 0   $, where $0<\delta <\frac{1}{3}$. }
\end{description}

For the sake of simplicity and without loss of generality, we assume that the functions $f\left(\cdot
\right),b\left(\cdot \right) $ and $\sigma \left(\cdot \right)$ are positive.

The Fisher information in this problem is
\begin{align*}
{\rm I}\left(\vartheta\right)=\int_{0}^{T} \frac{\dot S\left(\vartheta
  ,t\right)^2}{2S\left(\vartheta ,t\right)\sigma \left(t\right)} \;{\rm d}t.
\end{align*}
Define the function
\begin{align*}
G\left(\vartheta ,\vartheta _0\right)=\int_{0}^{T}\frac{\left[S\left(\vartheta
    ,t\right)-S\left(\vartheta_0 ,t\right)\right]^2}{4S\left(\vartheta
  ,t\right)\sigma \left(t\right) } \,{\rm d}t.
\end{align*}

{\it Conditions} {${\cal B}$.}
\begin{description}
\item[${\cal B}_1$.]{\it The Fisher information is  positive
\begin{align*}
\inf_{\vartheta \in\Theta }{\rm I}\left(\vartheta\right)>0.
\end{align*}
   }
\item[${\cal B}_2$.]{\it For any $\vartheta _0\in \Theta $ and $\nu >0$}
\begin{align*}
\inf_{\left|\vartheta -\vartheta _0\right|>\nu }G\left(\vartheta ,\vartheta _0\right)>0.
\end{align*}

\end{description}

 The family of measures $\left\{\Pb_{\vartheta }^{\left(\varepsilon
   \right)},\vartheta \in\Theta \right\}$  induced by the observation process
 $X^T=\left(X_t,0\leq t\leq
 T\right)$ on the space
    $\left({\cal C}\left[0,T\right],{\scr B}\right)$ of continuous functions on
 $\left[0,T\right]$ is locally asymptotically normal (LAN) (see Lemma \ref{L4}
 below). Therefore the  mean squared estimation error satisfies the Hajek-Le
 Cam's  minimax lower bound:
 for any $\vartheta _0\in\Theta $ and any estimator $\bar\vartheta _\varepsilon  $
\begin{align}
\label{2-5}
\lim_{\nu  \rightarrow 0}\Liminf_{\varepsilon \rightarrow
  0}\sup_{\left|\vartheta -\vartheta _0\right|\leq \nu  } \frac{\psi_\varepsilon}{\varepsilon
}\Ex_{\vartheta} \left|\bar\vartheta_\varepsilon
-\vartheta\right|^2\geq  {\rm I}\left(\vartheta_0 \right)^{-1}.
\end{align}
The estimator $\vartheta _\varepsilon ^*$ is called asymptotically efficient if
for any $\vartheta _0\in\Theta $
\begin{align*}
\lim_{\nu  \rightarrow 0}\lim_{\varepsilon \rightarrow
  0}\sup_{\left|\vartheta -\vartheta _0\right|\leq \nu } \frac{\psi_\varepsilon}{\varepsilon
}\Ex_{\vartheta} \left|\vartheta_\varepsilon ^*
-\vartheta\right|^2=  {\rm I}\left(\vartheta_0 \right)^{-1}.
\end{align*}
For the proof of a more general result see, e.g., \cite{IH81}.

The main result of this paper is the following theorem.

\medskip

\begin{theorem}
\label{T1}
 Assume that conditions ${\cal A}$ and ${\cal B}$ are satisfied. Then the
 MLE $\hat\vartheta_\varepsilon $ and BE $\tilde\vartheta_\varepsilon $ are
 consistent, asymptotically normal
\begin{align*}
& \sqrt{\frac{\psi_\varepsilon}{\varepsilon }}\left(\hat\vartheta_\varepsilon
  -\vartheta_0\right) \Longrightarrow \zeta \sim {\cal N}\left(0,{\rm
    I}\left(\vartheta_0 \right)^{-1}\right),\qquad
  \sqrt{\frac{\psi_\varepsilon}{\varepsilon
  }}\left(\tilde\vartheta_\varepsilon -\vartheta_0\right) \Longrightarrow
  \zeta ,
\end{align*}
the convergence of moments holds,
\begin{align*}
\left(\frac{\psi_\varepsilon}{\varepsilon
}\right)^{\frac{p}{2}}\Ex_{\vartheta _0} \left|\hat\vartheta_\varepsilon
-\vartheta_0\right|^p\rightarrow \Ex_{\vartheta _0} \left|\zeta
\right|^p,\quad \left(\frac{\psi_\varepsilon}{\varepsilon
}\right)^{\frac{p}{2}}\Ex_{\vartheta _0} \left|\tilde\vartheta_\varepsilon
-\vartheta_0\right|^p\rightarrow \Ex_{\vartheta _0} \left|\zeta \right|^p,
\end{align*}
for any $p>0$,
and  both estimators are asymptotically efficient.
\end{theorem}
 \begin{proof} The proof of this theorem is based on the general results of
Ibragimov and Khasminskii \cite{IH81}.

Let us denote $M\left(\vartheta ,t\right)=f\left(\vartheta,
t\right)m\left(\vartheta ,t\right)$, where $m\left(\vartheta ,t\right) $ is the
solution of  equation \eqref{2-3}. The method is based on asymptotics of the likelihood ratios
\begin{align*}
 L\left(\vartheta ,X^T\right)=\exp\left\{\int_{0}^{T}
\frac{M\left(\vartheta ,t\right)}{\varepsilon ^2\sigma(t)^2}\,{\rm d}X_t- \int_{0}^{T}
\frac{M\left(\vartheta ,t\right)^2}{2\varepsilon ^2\sigma(t)^2}\,{\rm d}t\right\},\;\quad  \vartheta \in\Theta =\left(\alpha
,\beta \right).
\end{align*}
Define the normalized likelihood ratio process
\begin{align*}
Z_\varepsilon \left(u\right)=\frac{L\left(\vartheta_0+\varphi _\varepsilon u
  ,X^T\right)}{L\left(\vartheta_0 ,X^T\right)}, \qquad u\in\UU_\varepsilon
=\left(\frac{\alpha -\vartheta _0}{\varphi _\varepsilon },\frac{\beta
  -\vartheta _0}{\varphi _\varepsilon } \right),
\end{align*}
where $\varphi _\varepsilon =\sqrt{\varepsilon /\psi _\varepsilon }  $.
Let us first sketch the main idea of the proof of asymptotic normality of the MLE
$\hat\vartheta _\varepsilon $ (Theorem 3.1.1 in \cite{IH81}). Suppose that we already proved
the weak convergence of the random process $Z_\varepsilon \left(\cdot \right)$
to the limit process $Z\left(\cdot \right)$ as $\varepsilon\to 0$, where
\begin{align*}
Z\left(u\right)=\exp\left\{u\Delta \left(\vartheta _0\right)-\frac{u^2}{2}{\rm
    I}\left(\vartheta_0 \right) \right\},\qquad u\in {\cal R},
\end{align*}
and  $\Delta \left(\vartheta _0\right)\sim {\cal N}\left(0,{\rm
    I}\left(\vartheta_0 \right)\right)$. Then for any $x\in \cal R$
\begin{align*}
&\Pb_{\vartheta _0}\left(\sqrt{\frac{\psi _\varepsilon }{\varepsilon
  }}\left(\hat\vartheta _\varepsilon -\vartheta _0\right)<x
  \right)=\Pb_{\vartheta _0}\left(\hat\vartheta _\varepsilon <\vartheta
  _0+\varphi _\varepsilon x \right)\\
&\qquad \quad =\Pb_{\vartheta _0}\left(
  \sup_{\vartheta < \vartheta _0+\varphi _\varepsilon x} L\left(\vartheta
  ,X^T\right) >\sup_{\vartheta \geq \vartheta _0+\varphi _\varepsilon x}
  L\left(\vartheta ,X^T\right) \right)\\
 &\qquad \quad =\Pb_{\vartheta
    _0}\left( \sup_{\vartheta < \vartheta _0+\varphi _\varepsilon x}
  \frac{L\left(\vartheta ,X^T\right)}{L\left(\vartheta_0 ,X^T\right)}
  >\sup_{\vartheta \geq \vartheta _0+\varphi _\varepsilon x}
  \frac{L\left(\vartheta ,X^T\right)}{L\left(\vartheta_0 ,X^T\right)} \right)\\
 &\qquad \quad =\Pb_{\vartheta    _0}\left( \sup_{u <  x}  Z_\varepsilon \left(u\right)
  >\sup_{u \geq   x}  Z_\varepsilon \left(u\right) \right)\\
 &\qquad \quad \longrightarrow \Pb_{\vartheta    _0}\left( \sup_{u <  x}  Z \left(u\right)
  >\sup_{u \geq   x}  Z \left(u\right) \right)=\Pb_{\vartheta    _0}\left(\frac{\Delta \left(\vartheta _0\right)}{{\rm
    I}\left(\vartheta_0 \right)}<x\right).
\end{align*}
Here we changed the variable so that  $\vartheta =\vartheta _0+\varphi _\varepsilon
u$. Note that the random function $Z\left(\cdot \right)$ has the unique maximum at
the point $\zeta =\Delta \left(\vartheta _0\right){\rm
    I}\left(\vartheta_0 \right)^{-1}\sim {\cal N}\left(0,{\rm
    I}\left(\vartheta_0 \right)^{-1}\right) $.

Similar calculations show that the BE $\tilde\vartheta _\varepsilon $  is asymptotically  normal with the same
limit variance (see details in \cite{IH81}).

For the model under consideration
\begin{align*}
\ln Z_\varepsilon \left(u\right)=\int_{0}^{T}
\frac{M\left(\vartheta_u ,t\right)-M\left(\vartheta_0
  ,t\right)}{\varepsilon ^2\sigma(t)^2}\,{\rm d}X_t- \int_{0}^{T}
\frac{\left[M\left(\vartheta_u ,t\right)-M\left(\vartheta_0
    ,t\right)\right]^2}{2\varepsilon ^2\sigma(t)^2}\,{\rm d}t,
\end{align*}
where $\vartheta_u = \vartheta_0+\varphi _\varepsilon u$. The
formal Taylor formula yields
\begin{align*}
M\left(\vartheta_0+\varphi _\varepsilon u ,t\right)-M\left(\vartheta_0
,t\right)=\varphi _\varepsilon u \dot M\left(\vartheta _0,t\right)+
o\left(\varphi _\varepsilon \right) .
\end{align*}
Recall that  $\dot M\left(\vartheta ,t\right) $ is the
derivative in  $\vartheta $ of the random process $M\left(\vartheta
,t\right)$,
\begin{align}
\label{2-6}
\dot M\left(\vartheta ,t\right)=\dot f\left(\vartheta ,t\right)m\left(\vartheta
,t\right)+f\left(\vartheta, t\right)\dot m\left(\vartheta ,t\right).
\end{align}
Thus we need to find the asymptotics of the derivative $\dot m\left(\vartheta ,t\right) $ of
$ m\left(\vartheta ,t\right)$ and the related derivative $\dot \gamma_*
\left(\vartheta ,t\right)$ of $\gamma_* \left(\vartheta ,t\right)$, since the
equation \eqref{2-3} for $m\left(\vartheta ,t\right)$ depends on $\gamma_*
\left(\vartheta ,t\right)$.

Let us denote
\begin{align*}
 \gamma_0 \left(\vartheta
  ,t\right)&=\frac{b\left(\vartheta ,t\right)\sigma (t)}{f\left(\vartheta
  ,t\right)},\quad A\left(\vartheta ,t\right)=\frac{\gamma_* \left(\vartheta
  ,t\right)f\left(\vartheta,
  t\right)}{\sigma(t)^2} ,\\
 q_\varepsilon \left(\vartheta ,t\right)&=a\left(\vartheta ,t\right)+
  \frac{\psi_\varepsilon}{\varepsilon }A\left(\vartheta
  ,t\right)f\left(\vartheta ,t\right).
\end{align*}
Then equations \eqref{2-3}, \eqref{2-4} take the form
\begin{align}
\label{2-7}
&{\rm d}m\left(\vartheta ,t\right)=-q_\varepsilon \left(\vartheta
,t\right)m\left(\vartheta ,t\right){\rm d}t + \frac{\psi_\varepsilon}{\varepsilon }
A\left(\vartheta
  ,t\right) {\rm d}X_t, \\
\label{2-8}
&\frac{\partial \gamma _*\left(\vartheta ,t\right)}{\partial t}=
-2a(\vartheta,t)\gamma_* \left(\vartheta ,t\right)
-\frac{\psi_\varepsilon}{\varepsilon } A\left(\vartheta ,t\right)^2 {
  \sigma(t)^2 } +\frac{\psi_\varepsilon}{\varepsilon
}b(\vartheta,t)^2,
\end{align}
with the initial values $m\left(\vartheta ,0\right)= 0, \gamma _*\left(\vartheta
,0\right)=0$.

\begin{lemma}
\label{L1} Let the conditions ${\cal A}_1,{\cal A}_2$ be satisfied and
$\frac{\varepsilon }{\psi _\varepsilon }\rightarrow 0$. Then, for any $t_0\in
(0,T]$,
\begin{align}
\label{2-9}
\sup_{t_0\leq t\leq T}\left|\gamma_* \left(\vartheta ,t\right) -\gamma_0
\left(\vartheta ,t\right)\right|=O\left(\frac{ \varepsilon
}{\psi_\varepsilon}\right).
\end{align}
\end{lemma}
\begin{proof}
Introduce the additional  Riccati  equation
\begin{align}
\label{2-10}
\frac{\partial \hat\gamma\left(t\right)}{\partial t}=
-2a_m\hat \gamma \left(t\right) -\frac{\psi
  _\varepsilon}{\varepsilon } \frac{\hat\gamma
  \left(t\right)^2f_m^2}{S^2} +\frac{\psi _\varepsilon}{\varepsilon }B^2, \qquad \hat\gamma\left(0\right)=0 ,
\end{align}
where we denoted
\begin{align*}
a_m=\inf_{0\leq t\leq T}\inf_{\vartheta \in\Theta }\left|a\left(\vartheta
,t\right)\right| ,\qquad f_m=\inf_{0\leq t\leq T}\inf_{\vartheta \in\Theta
} f\left(\vartheta ,t\right) ,\quad  \sigma_M=\sup_{0\leq t\leq
  T} \sigma \left(t\right).
\end{align*}
Note that by condition ${\cal A}_2$ we have $f_m>0$. Due to the
comparison theorem for ordinary differential equations, $ \gamma_* \left(\vartheta ,t\right)\leq \hat
\gamma \left(t\right)$ holds  for all $t\in \left[0,T\right]$. This bound is intuitive, i.e., as replacing
the ``noise coefficient'' $\sigma\left(t\right)$ by its maximal value should increase the estimation
error.
Further, as we take smaller coefficient $f_m $ in the drift $f\left(\vartheta
,t\right)Y_t$ the error increases as well. The solution of  Riccati equation \eqref{2-10} with
constant coefficients can be written explicitly (see \cite{A83})
\begin{align*}
\hat\gamma \left(t\right)=e^{-2R_\varepsilon t}\left[\frac{\psi
    _\varepsilon f_m^2\left(1-e^{-2R_\varepsilon
      t}\right)}{2\varepsilon R_\varepsilon \sigma_M^2}-\hat\gamma
  ^{-1}\right]^{-1}+\hat\gamma
\end{align*}
Here
\begin{align*}
R_\varepsilon =\left(a_m^2+\frac{\psi _\varepsilon^2}{\varepsilon
  ^2}\frac{B^2f_m^2}{\sigma_M^2}\right)^{1/2} =\frac{\psi _\varepsilon}{\varepsilon }\frac{f_mB}{\sigma_M}\left(1+O\left(\frac{\varepsilon^2}{\psi
 _\varepsilon^2} \right)\right)
\end{align*}
and
\begin{align*}
\hat\gamma =\frac{\varepsilon}{\psi_\varepsilon}\frac{
  a_m\sigma_M^2}{f_m^2}\left[\left(1+\frac{\psi_\varepsilon^2}{\varepsilon
    ^2}\frac{B^2f_m^2}{a_m^2\sigma_M^2}\right)^{1/2}-1\right]=\frac{B\sigma_M}{f_m}\left(1+O\left(\frac{\varepsilon}{\psi
  _\varepsilon} \right)\right).
\end{align*}
Therefore for any $t_0\in(0,T]$
\begin{align}
\label{2-11}
\sup_{t_0\leq t\leq T}\left|\hat\gamma
\left(t\right)-\frac{B\sigma_M}{f_m}\right|=O\left(\frac{\varepsilon}{\psi
 _\varepsilon}\right).
\end{align}
Let us write the Riccati equation in integral form
\begin{align*}
\frac{\varepsilon}{\psi
  _\varepsilon}&\left[\gamma_* \left(\vartheta ,t\right)-\gamma_* \left(\vartheta
,t_0\right)+2\int_{t_0}^{t}a\left(\vartheta ,s\right)\gamma_* \left(\vartheta
,s\right){\rm d}s\right]\\
&\qquad \qquad  =-\int_{t_0}^{t}\frac{\gamma _*\left(\vartheta
  ,s\right)^2f\left(\vartheta ,s\right)^2}{\sigma \left(s\right)^2}{\rm d}s
+\int_{t_0}^{t}b\left(\vartheta ,s\right)^2{\rm d}s .
\end{align*}
As $\gamma_* \left(\vartheta ,t\right)\leq \hat\gamma \left(t\right) $ and
$\hat\gamma \left(t\right) $ according to \eqref{2-11} is bounded,  the left hand side of
the integral equation tends to 0. Hence
\begin{align*}
\int_{t_0}^{t}\frac{\gamma _*\left(\vartheta
  ,s\right)^2f\left(\vartheta ,s\right)^2}{\sigma \left(s\right)^2}{\rm d}s
=\int_{t_0}^{t}b\left(\vartheta ,s\right)^2{\rm d}s+O\left(\frac{\varepsilon}{\psi
  _\varepsilon}\right)
\end{align*}
and
\begin{align*}
\gamma _*\left(\vartheta
  ,t\right)^2
=\frac{b\left(\vartheta ,t\right)^2\sigma \left(t\right)^2}{ f\left(\vartheta
  ,t\right)^2} +O\left(\frac{\varepsilon}{\psi
  _\varepsilon}\right)=\gamma _0\left(\vartheta
  ,t\right)^2+O\left(\frac{\varepsilon}{\psi
  _\varepsilon}\right).
\end{align*}
\end{proof}
It can be checked that all terms of order $O\left(\frac{\varepsilon }{\psi
  _\varepsilon }\right)$ above satisfy
\begin{align*}
\left|O\left(\frac{\varepsilon }{\psi _\varepsilon }\right)\right|\leq C\,\frac{\varepsilon }{\psi _\varepsilon },
\end{align*}
where the constant  $C>0$ can be chosen independent on $\vartheta $.

Let us denote
\begin{align*}
\dot \gamma _0\left(\vartheta ,t\right)=\frac{  b(\vartheta,t)\sigma \left(t\right)}{f(\vartheta,t)}
\frac{\partial }{\partial \vartheta }   \left[\ln \frac{
    b(\vartheta,t)}{ f(\vartheta,t)}\right].
\end{align*}
\begin{lemma}
\label{L2}
Let the conditions ${\cal A}_1-{\cal A}_3$ be satisfied and $\frac{\varepsilon
}{\psi _\varepsilon }\rightarrow 0$, then, for any $t_0\in (0,T]$,
\begin{align}
\label{2-12}
\sup_{t_0\leq t\leq T}\left|\dot \gamma _*\left(\vartheta ,t\right)-\dot \gamma _0\left(\vartheta ,t\right)\right|=O\left(\frac{\varepsilon}{\psi
  _\varepsilon}\right)
\end{align}
\end{lemma}
\begin{proof}
The derivative $\dot \gamma _*\left(\vartheta ,t\right)$ solves the
equation
\begin{align}
\frac{\partial \dot \gamma_* \left(\vartheta ,t\right)}{\partial t }&=
-2q_\varepsilon \left(\vartheta ,t\right) \dot \gamma_*
\left(\vartheta ,t\right) +2\frac{\psi _\varepsilon}{\varepsilon
}b(\vartheta,t)\dot b(\vartheta,t)\nonumber\\
&\qquad  -2\left[\dot a\left(\vartheta ,t\right)   +\frac{\psi_\varepsilon
  }{\varepsilon } A\left(\vartheta  ,t\right) \dot f\left( \vartheta,
  t\right)\right] \gamma_*\left(\vartheta ,t\right)
\label{2-13}
\end{align}
subject to $\dot \gamma_* \left(\vartheta ,0\right)=0  $, which can be justified by the usual argument,
using the Gronwall-Bellman lemma.
This is linear equation
and its solution can be written as follows
\begin{align*}
 &\dot \gamma_* \left(\vartheta ,t\right)=-2\int_{0}^{t}e^{-2\int_{s}^{t}q_\varepsilon \left(\vartheta ,v\right){\rm
     d}v} \dot a\left(\vartheta ,s\right)\gamma _*\left(\vartheta ,s\right) {\rm d}s\\
 &\qquad +
\frac{2\psi_\varepsilon}{\varepsilon } \int_{0}^{t}e^{-2\int_{s}^{t}q_\varepsilon
   \left(\vartheta ,v\right){\rm d}v} \left[ b(\vartheta,s)\dot
   b(\vartheta,s)-\frac{\gamma _*\left(\vartheta ,s\right)^2f\left( \vartheta,
   s\right) \dot f\left( \vartheta,
   s\right)}{\sigma \left(s\right)^2}\right]{\rm d}s .
\end{align*}
By Lemma 1 in \cite{Kut19a} and \eqref{2-9} we have the representation
\begin{align*}
\dot \gamma_* \left(\vartheta ,t\right)&=\frac{ b(\vartheta,t)\dot
  b(\vartheta,t)\sigma \left(t\right)^2-\gamma _*\left(\vartheta
  ,t\right)^2f\left( \vartheta, t\right) \dot f\left( \vartheta,
  t\right)}{\gamma _*\left(\vartheta ,t\right) f\left( \vartheta, t\right)^2
}+O\left(\frac{\varepsilon}{\psi _\varepsilon}\right)\\
&= \frac{  b(\vartheta,t)\sigma \left(t\right)}{f(\vartheta,t)}\left[\frac{\dot
    b(\vartheta,t)}{ b(\vartheta,t)}-\frac{\dot
    f(\vartheta,t)}{ f(\vartheta,t)}\right]+O\left(\frac{\varepsilon}{\psi
  _\varepsilon}\right)\\
&= \frac{  b(\vartheta,t)\sigma \left(t\right)}{f(\vartheta,t)}
\frac{\partial }{\partial \vartheta }   \left[\ln \frac{
    b(\vartheta,t)}{ f(\vartheta,t)}\right]+O\left(\frac{\varepsilon}{\psi _\varepsilon}\right).
\end{align*}

\end{proof}

As before, the terms $O\left(\frac{\varepsilon}{\psi
  _\varepsilon}\right) $ can be shown to satisfy
\begin{align*}
\left|O\left(\frac{\varepsilon}{\psi _\varepsilon}\right)\right|\leq
C\;\frac{\varepsilon}{\psi _\varepsilon},
\end{align*}
with a constant  $C>0$ independent of $\vartheta $.

Below $O_p\left(\frac{\varepsilon }{\psi _\varepsilon }\right)$ and
$o_p\left(\frac{\varepsilon }{\psi _\varepsilon }\right)$    means that for
any p>1
\begin{align*}
\Ex_{\vartheta _0}\left|O_p\left(\frac{\varepsilon }{\psi _\varepsilon
}\right)\right|^p\leq C\left(\frac{\varepsilon }{\psi _\varepsilon
}\right)^p,\qquad \Ex_{\vartheta _0}\left|o_p\left(\frac{\varepsilon }{\psi
  _\varepsilon }\right)\right|^p\leq c_\varepsilon \left(\frac{\varepsilon
}{\psi _\varepsilon }\right)^p,\quad c_\varepsilon\rightarrow 0,
\end{align*}
where the constants $C>0$ and $c_\varepsilon >0$ do not depend on $\vartheta $.
Recall that $S\left(\vartheta ,t\right)=f\left(\vartheta ,t\right)b\left(\vartheta ,t\right)$.
\begin{lemma}
\label{L3} Let conditions ${\cal A}$ be satisfied, then
\begin{align}
\label{2-14}
\dot M\left(\vartheta _0,t\right)&= \sqrt{\varepsilon \psi_\varepsilon} \sqrt{\frac{\sigma \left(t\right) }{2S\left(\vartheta_0 ,t\right)}
} \; {\dot S\left(\vartheta_0,t\right)}{ }\; \xi _{t,\varepsilon}
\left(1+o\left(1\right)\right)+O_p\left(\frac{\varepsilon }{\psi
  _\varepsilon}\right),
\end{align}
with $\xi _{t,\varepsilon }\left(\vartheta \right)\;\Longrightarrow \;\xi _t\sim
 {\cal N}\left(0,1\right). $
Here $\xi _t,t\in (0,T]$ are mutually  independent random variables.
\end{lemma}
\begin{proof}
The formal derivative  of equation \eqref{2-7} with respect to $\vartheta_0 $   gives  the
equation   for    derivative $\dot m\left(\vartheta_0 ,t\right)$,
\begin{align}
{\rm d}\dot m\left(\vartheta_0 ,t\right)&=-q_\varepsilon \left(\vartheta_0 ,t\right)\dot m\left(\vartheta_0
  ,t\right){\rm d}t+ \frac{\psi_\varepsilon
    }{\varepsilon }\dot A\left(\vartheta_0
  ,t\right){\rm d}X_t \nonumber\\
 & \quad-\left[\dot a(\vartheta_0,t)+ \frac{\psi_\varepsilon
    }{\varepsilon } F\left(\vartheta_0 ,t\right)\right] m\left(\vartheta_0
  ,t\right){\rm d}t ,\quad \dot m\left(\vartheta_0 ,0\right)=0,
\label{2-15}
\end{align}
where   we denoted
\begin{align*}
 F\left(\vartheta_0 ,t\right)&= \dot A\left(\vartheta_0
  ,t\right)f\left(\vartheta_0 ,t\right)+A\left(\vartheta_0
  ,t\right)\dot f\left(\vartheta_0 ,t\right)  .
\end{align*}
To verify that $m\left(\vartheta_0 ,t\right)$ has derivative in the mean square, we
can write the equation for $m\left(\vartheta_0 +h,t\right)$ and then consider the
equation for the
difference $m\left(\vartheta_0 +h,t\right)-m\left(\vartheta_0,t\right)-h\dot
m\left(\vartheta_0 ,t\right)$. Using Gronwall-Bellman lemma we obtain the
estimate
 \begin{align*}
\Ex_{\vartheta_0}\left[m\left(\vartheta_0 +h,t\right)-m\left(\vartheta_0,t\right)-h\dot
m\left(\vartheta_0 ,t\right) \right]^2=o\left(h^2\right).
\end{align*}
Due to the {\it Innovation Theorem}, the observation process
can be written as
\begin{align*}
{\rm d}X_t=f\left(\vartheta_0,t\right)m\left(\vartheta_0,t\right){\rm
  d}t+\varepsilon \sigma \left(t\right) {\rm d}\bar W_t,\qquad 0\leq t\leq T.
\end{align*}
 Substitution of
this differential in \eqref{2-15} gives  the equation
\begin{align*}
{\rm d}\dot m\left(\vartheta_0 ,t\right)&=-q_\varepsilon \left(\vartheta_0
,t\right)\dot m\left(\vartheta_0
  ,t\right){\rm d}t- \frac{\psi_\varepsilon
    }{\varepsilon }   A\left(\vartheta_0 ,t\right)\dot  f\left(\vartheta_0
  ,t\right)m\left(\vartheta_0
  ,t\right){\rm d}t \\
 &\quad  -\dot a(\vartheta_0,t) m\left(\vartheta_0
  ,t\right){\rm d}t +\psi_\varepsilon
    \dot A\left(\vartheta_0
  ,t\right)\sigma \left(t\right){\rm d}\bar W_t,\quad \dot m\left(\vartheta_0 ,0\right)=0.
\end{align*}
Hence we have
\begin{align}
\label{ku}
\dot m\left(\vartheta_0
,t\right)&=\psi _\varepsilon\int_{0}^{t}e^{-\int_{s}^{t}q_\varepsilon \left(\vartheta_0
  ,v\right){\rm d}v} \dot A\left(\vartheta_0 ,s\right)\sigma \left(s\right) {\rm
  d}\bar W_s\nonumber\\
&\qquad -\frac{\psi _\varepsilon}{\varepsilon }\int_{0}^{t}e^{-\int_{s}^{t}q_\varepsilon \left(\vartheta_0
  ,v\right){\rm d}v} A\left(\vartheta_0 ,s\right)\dot  f\left(\vartheta_0
    ,s\right)m\left(\vartheta_0
  ,s\right){\rm d}s\nonumber\\
&\qquad -\int_{0}^{t}e^{-\int_{s}^{t}q_\varepsilon \left(\vartheta_0
  ,v\right){\rm d}v} \dot a(\vartheta_0,s) m\left(\vartheta_0
  ,s\right){\rm d}s.
\end{align}
The main contribution in these integrals is due to the values at
the vicinity of the point $t$. Therefore the  functions $A\left(\vartheta _0,s\right),
f\left(\vartheta _0,s\right),\sigma \left(s\right)$ in these integrals we can
replaced by the values $A\left(\vartheta _0,t\right), f\left(\vartheta
_0,t\right),\sigma \left(t\right)$. Here we use the Taylor expansion such as
$A\left(\vartheta _0,s\right)=A\left(\vartheta
_0,t\right)+\left(s-t\right)A'\left(\vartheta _0,\tilde s\right) $.

By Lemma 1 in \cite{Kut19a}, these integrals  satisfy
\begin{align*}
&\frac{\psi _\varepsilon}{\varepsilon }\int_{0}^{t}e^{-\int_{s}^{t}q_\varepsilon \left(\vartheta_0
  ,v\right){\rm d}v}  A\left(\vartheta_0 ,s\right)\dot f\left(\vartheta_0
   ,s\right)m\left(\vartheta_0
  ,s\right){\rm d}s\\
&\qquad =\frac{\psi _\varepsilon}{\varepsilon }\int_{0}^{t}e^{-\int_{s}^{t}q_\varepsilon \left(\vartheta_0
  ,v\right){\rm d}v} A\left(\vartheta_0 ,s\right)\dot   f\left(\vartheta_0
    ,s\right) \left[ m\left(\vartheta_0
  ,s\right)- m\left(\vartheta_0
  ,t\right)\right]{\rm d}s\\
&\qquad \qquad +\frac{\psi _\varepsilon}{\varepsilon }\int_{0}^{t}e^{-\int_{s}^{t}q_\varepsilon \left(\vartheta_0
  ,v\right){\rm d}v}  A\left(\vartheta_0 ,s\right)\dot f\left(\vartheta_0
    ,s\right){\rm d}s\;m\left(\vartheta_0
  ,t\right)\\
&\qquad =I_\varepsilon +     \frac{\dot  f\left(\vartheta_0
    ,t\right)m\left(\vartheta_0
  ,t\right) }{f\left(\vartheta_0
    ,t\right)}+O_p\left(\frac{\varepsilon }{\psi_\varepsilon}\right).
\end{align*}
where  $I_\varepsilon $ denotes the integral with $
m_{t,s} =m\left(\vartheta_0 ,t\right)- m\left(\vartheta_0 ,s\right) $.

Similarly
\begin{align*}
&\int_{0}^{t}e^{-\int_{s}^{t}q_\varepsilon \left(\vartheta_0
  ,v\right){\rm d}v} \dot a(\vartheta_0,s) m\left(\vartheta_0
  ,s\right){\rm d}s       \\
&\quad =  \int_{0}^{t}e^{-\int_{s}^{t}q_\varepsilon \left(\vartheta_0
  ,v\right){\rm d}v} \dot a(\vartheta_0,s) \left[m\left(\vartheta_0
  ,s\right)- m\left(\vartheta_0
  ,t\right)\right]{\rm d}s+   O_p\left(\frac{\varepsilon }{\psi_\varepsilon}\right).
\end{align*}

Changing integration  variable to  $s=t-\frac{r\varepsilon }{\psi _\varepsilon}$ gives
\begin{align}
\label{2-16}
m_{t,s}&=-\int_{s}^{t}a\left(\vartheta _0,v\right)m\left(\vartheta
_0,v\right){\rm d}v+\psi _\varepsilon
\int_{s}^{t}A\left(\vartheta _0,v\right)\sigma \left(v\right){\rm d}\bar W_v\nonumber\\
&=-\frac{r\varepsilon }{\psi _\varepsilon}a\left(\vartheta _0,t\right)m\left(\vartheta
_0,t\right)\left(1+o_p\left(1\right)\right)\nonumber\\
&\qquad \qquad +\psi _\varepsilon
A\left(\vartheta _0,t\right)\sigma \left(t\right) \left[\bar W_t-\bar
  W_{t-\frac{r\varepsilon }{\psi_\varepsilon}}\right]\left(1+o\left(1\right)\right) .
\end{align}
Changing the integration  variables and letting $
k\left(t\right)=A\left(\vartheta _0,t\right) f\left(\vartheta _0,t\right) $ we also get
\begin{align*}
&\frac{\psi _\varepsilon}{\varepsilon }\int_{0}^{t}e^{-\int_{s}^{t}q_\varepsilon \left(\vartheta_0
  ,v\right){\rm d}v} A\left(\vartheta_0 ,s\right)\dot   f\left(\vartheta_0
    ,s\right) \left[ m\left(\vartheta_0
  ,s\right)- m\left(\vartheta_0
  ,t\right)\right]{\rm d}s\\
&\qquad =A\left(\vartheta_0
  ,t\right)\dot   f\left(\vartheta_0
    ,t\right)\int_{0}^{\frac{t\psi _\varepsilon}{\varepsilon }}e^{- r k\left(t\right) }
  \left[- \frac{r\varepsilon }{\psi _\varepsilon}
    a\left(\vartheta _0,t\right) m\left(\vartheta_0
  ,t\right)\right.\\
&\qquad \qquad\qquad  \left.+\sqrt{\varepsilon \psi _\varepsilon}A\left(\vartheta _0,t\right)\sigma
    \left(t\right)w_\varepsilon \left(r\right)  \right]      {\rm
    d}r\left(1+o\left(1\right)\right)\\
&\qquad =A\left(\vartheta_0
  ,t\right)^2\dot   f\left(\vartheta_0
    ,t\right)\sigma \left(t\right)  \sqrt{\varepsilon \psi_\varepsilon}   \int_{0}^{\frac{t\psi _\varepsilon}{\varepsilon
  }}e^{- r k\left(t\right) } w_\varepsilon \left(r\right)      {\rm
    d}r\left(1+o\left(1\right)\right)  ,
\end{align*}
where $w_\varepsilon \left(r\right)= \sqrt{\frac{\varepsilon }{\psi
    _\varepsilon}} \left[\bar W_t-\bar W_{t-\frac{r\varepsilon }{\psi
          _\varepsilon}}\right]    $ is a Wiener process.
Further,
\begin{align*}
\int_{0}^{\frac{t\psi _\varepsilon}{\varepsilon }}e^{- r
  k\left(t\right) } w_\varepsilon \left(r\right) {\rm
  d}r&=\frac{1}{k\left(t\right)}\int_{0}^{\frac{t\psi \left(\varepsilon
    \right)}{\varepsilon }
}\left[e^{-k\left(t\right)z}-e^{-tk\left(t\right)\frac{\psi \left(\varepsilon
      \right)}{\varepsilon }} \right] {\rm d} w_\varepsilon
\left(z\right)\\
&= \frac{1}{k\left(t\right)^{3/2}}
\int_{0}^{\frac{tk\left(t\right)\psi _\varepsilon}{\varepsilon }
}e^{-y} \;{\rm d} \tilde w_\varepsilon
\left(y\right)\;\left(1+o_p\left(1\right)\right)\\
 &=\frac{1}{A\left(\vartheta _0,t\right)^{3/2}f\left(\vartheta_0
    ,t\right)^{3/2}\sqrt{2}} \;\xi _{t,\varepsilon
}\left(1+o_p\left(1\right)\right)
\end{align*}
Here $\tilde w_\varepsilon
\left(y\right)=k\left(t\right)^{1/2} w_\varepsilon
\left(y/k\left(t\right)\right) $ and
\begin{align*}
\xi _{t,\varepsilon }=\sqrt{2}\int_{0}^{\frac{tk\left(t\right)\psi
    _\varepsilon}{\varepsilon }
}e^{-y} \;{\rm d} \tilde w_\varepsilon
\left(y\right)\Longrightarrow \xi _t\;\sim \;{\cal N}\left(0,1\right),
\end{align*}
where $\xi _t,t\in(0,T]$ are mutually independent random variables.

For the stochastic integral similar argument implies
\begin{align*}
&\sqrt{\frac{\psi _\varepsilon}{\varepsilon
  }}\int_{0}^{t}e^{-\int_{s}^{t}q_\varepsilon \left(\vartheta_0
  ,v\right){\rm d}v} \dot A\left(\vartheta_0 ,s\right)\sigma \left(s\right) {\rm
  d}\bar W_s\\
&\qquad \qquad =\dot A\left(\vartheta_0 ,t\right)\sigma \left(t\right)
  \sqrt{\frac{\psi _\varepsilon}{\varepsilon }}\int_{0}^{t}
  e^{-\frac{\psi _\varepsilon}{\varepsilon }k\left(t\right)
    \left(t-s\right)}{\rm d}\bar W_s\left(1+o_p\left(1\right)\right)\\
&\qquad \qquad =\frac{\dot A\left(\vartheta_0 ,t\right)\sigma
    \left(t\right)}{\sqrt{k\left(t\right)}}
  \int_{0}^{\frac{tk\left(t\right)\psi _\varepsilon}{\varepsilon }}
  e^{-y}\;{\rm d}\tilde w_\varepsilon \left(y\right)\left(1+o_p\left(1\right)\right)\\
&\qquad \qquad =\frac{\dot A\left(\vartheta_0 ,t\right)\sigma
    \left(t\right)}{\sqrt{2A_0\left(\vartheta_0 ,t\right)f\left(\vartheta_0
    ,t\right)}} \;  \xi _{t,\varepsilon }\,\left(1+o_p\left(1\right)\right).
\end{align*}

Define the limit  functions, $\varepsilon \rightarrow 0$,
\begin{align*}
A_0\left(\vartheta_0 ,t\right)&=\frac{\gamma _0\left(\vartheta_0
  ,t\right)f\left(\vartheta_0 ,t\right)}{\sigma
  \left(t\right)^2}=\frac{b\left(\vartheta_0 ,t\right)}{\sigma
  \left(t\right)},\\
 \dot
A_0\left(\vartheta_0 ,t\right)&=\frac{\dot \gamma _0\left(\vartheta_0
  ,t\right)f\left(\vartheta_0 ,t\right)+\gamma _0\left(\vartheta_0
  ,t\right)\dot f\left(\vartheta_0 ,t\right)}{\sigma
  \left(t\right)^2}=\frac{\dot b\left(\vartheta_0 ,t\right)}{\sigma
  \left(t\right)}.
\end{align*}
Then
\begin{align*}
&\frac{A_0\left(\vartheta_0 ,t\right)^2\dot f\left(\vartheta _0,t\right)\sigma
  \left(t\right)}{A_0\left(\vartheta_0 ,t\right)^{3/2}f\left(\vartheta
  _0,t\right)^{3/2} }+\frac{\dot A_0\left(\vartheta_0 ,t\right)\sigma
  \left(t\right)}{A_0\left(\vartheta_0 ,t\right)f\left(\vartheta_0 ,t\right)}\\
&\qquad \qquad \qquad =\sqrt{\frac{b\left(\vartheta_0 ,t\right)\sigma \left(t\right)
  }{f\left(\vartheta_0 ,t\right)}}\left( \frac{\dot f\left(\vartheta
  _0,t\right)}{ f\left(\vartheta _0,t\right)}+ \frac{\dot b\left(\vartheta
  _0,t\right)}{ b\left(\vartheta _0,t\right)}\right).
\end{align*}

All this allows us to write the expansion for the derivative
\begin{align*}
\dot m\left(\vartheta _0,t\right)&=-\frac{\dot f\left(\vartheta
  _0,t\right)m\left(\vartheta _0,t\right)}{f\left(\vartheta _0,t\right)}
+O\left(\frac{\varepsilon }{\psi _\varepsilon}\right)\\
&\quad  +\sqrt{\varepsilon \psi_\varepsilon}\sqrt{\frac{b\left(\vartheta_0 ,t\right)\sigma \left(t\right)
  }{2f\left(\vartheta_0 ,t\right)}}\left( \frac{\dot f\left(\vartheta
  _0,t\right)}{ f\left(\vartheta _0,t\right)}+ \frac{\dot b\left(\vartheta
  _0,t\right)}{ b\left(\vartheta _0,t\right)}\right)\xi _{t,\varepsilon
}\left(1+o\left(1\right)\right).
\end{align*}
Now \eqref{2-14} follows from this representation and \eqref{2-6}.

\end{proof}

Let us now return to the normalized likelihood ratio $Z_\varepsilon
\left(u\right)$.
\begin{lemma}
\label{L4} Let the conditions ${\cal A}$ and ${\cal B}_1$ be satisfied, then the
family of measures $\left\{\Pb_{\vartheta }^{\left(\varepsilon
  \right)},\vartheta \in \Theta \right\}$ is locally asymptotically normal in
$\Theta $, i.e., the normalized likelihood ratio $Z_\varepsilon \left(u\right)
$ for all $\vartheta _0\in\Theta $  admits the representation
\begin{align*}
Z_\varepsilon \left(u\right)=\exp\left\{u\Delta _\varepsilon \left(\vartheta
_0,X^T\right)-\frac{u^2}{2}{\rm I}\left(\vartheta _0\right)+r_\varepsilon  \right\},
\end{align*}
where
 $$
\Delta _\varepsilon \left(\vartheta_0,X^T\right)=  \int_{0}^{T}\frac{\dot
  S\left(\vartheta_0 ,t\right)\,\xi _{t,\varepsilon }}{\sqrt{2S\left(\vartheta_0 ,t\right)\sigma
    \left(t\right)}}  {\rm d}\bar W_\varepsilon \left(t\right) \Longrightarrow
    \Delta  \left(\vartheta_0\right)\sim   {\cal N}\left(0,{\rm I}\left(\vartheta _0\right)\right).
 $$
 Here $r_\varepsilon \rightarrow 0$ and $\xi _{t,\varepsilon }\Rightarrow \xi
 _t$ for the random variables $ \xi _t,t\in (0,T]$.
\end{lemma}
\begin{proof}
The process $Z_\varepsilon \left(u\right)$ admits of the representation
\begin{align*}
\ln Z_\varepsilon \left(u\right)&=\frac{u\varphi _\varepsilon }{\varepsilon }\int_{0}^{T}\dot
M\left(\vartheta _0,t\right){\rm d}\bar W_t  -\frac{u^2\varphi _\varepsilon
  ^2}{2\varepsilon ^2}\int_{0}^{T}\dot
M\left(\vartheta _0,t\right)^2{\rm d}t+o_p\left(1\right)\\
&  =u\int_{0}^{T} \frac{\dot
  S\left(\vartheta _0,t\right)\xi _{t,\varepsilon }}{\sqrt{2S\left(\vartheta_0
  ,t\right)\sigma \left(t\right)}}{\rm
  d}\bar W_t-\frac{u^2}{2}\int_{0}^{T} \frac{\dot
  S\left(\vartheta _0,t\right)^2\xi _{t,\varepsilon }^2}{2S\left(\vartheta_0
  ,t\right)\sigma \left(t\right)}{\rm
  d}t +o_p\left(1\right) .
\end{align*}

Let us denote
\begin{align*}
N_\varepsilon =\sqrt{\frac{\psi _\varepsilon}{\varepsilon
}}\int_{0}^{T}R\left(t\right)\left(\xi_{t,\varepsilon }^2-1\right){\rm d}t
\end{align*}
where $R\left(\cdot \right) $ is some bounded function. As in the proof of
Lemma 4 in \cite{Kut19a},  for any $n>0$,
\begin{align}
\label{2-17}
 \Ex_{\vartheta _0}\left|N_\varepsilon \right|^{2n}\leq C.
\end{align}
Therefore we can write
\begin{align*}
\Ex_{\vartheta _0} \left(\int_{0}^{T} \frac{\dot
  S\left(\vartheta _0,t\right)^2\xi _{t,\varepsilon }^2}{2S\left(\vartheta_0
  ,t\right)\sigma \left(t\right)}\;{\rm
  d}t-\int_{0}^{T} \frac{\dot
  S\left(\vartheta _0,t\right)^2}{2S\left(\vartheta_0
  ,t\right)\sigma \left(t\right)}\;{\rm
  d}t\right)^2\leq C\,\frac{\varepsilon }{\psi _\varepsilon}
\end{align*}
and obtain  the convergence  in probability
\begin{align*}
\int_{0}^{T} \frac{\dot
  S\left(\vartheta _0,t\right)^2\xi _{t,\varepsilon }^2}{2S\left(\vartheta_0
  ,t\right)\sigma \left(t\right)}{\rm
  d}t\left(1+o\left(1\right)\right)\longrightarrow \int_{0}^{T} \frac{\dot
  S\left(\vartheta _0,t\right)^2}{2S\left(\vartheta_0 ,t\right)\sigma
  \left(t\right)}\;{\rm d}t={\rm I}\left(\vartheta _0\right).
\end{align*}
This limit in probability allows us to apply the central limit theorem for the
stochastic integral and to obtain the convergence
\begin{align*}
\int_{0}^{T} \frac{\dot
  S\left(\vartheta _0,t\right)\xi _{t,\varepsilon }}{\sqrt{2S\left(\vartheta_0
  ,t\right)\sigma \left(t\right)}}\;{\rm
  d}\bar W_t\Longrightarrow {\cal N}\left(0,{\rm I}\left(\vartheta _0\right)
\right).
\end{align*}
\end{proof}
Let us denote
\begin{align*}
G_\varepsilon \left(\vartheta ,\vartheta _0\right)=- \frac{\varepsilon}{ \psi
  _\varepsilon} \ln \frac{L\left(\vartheta ,X^T\right)}{L\left(\vartheta_0
  ,X^T\right)}.
\end{align*}

\begin{lemma}
\label{L5}
Assume that the conditions ${\cal A}_1,{\cal A}_2,{\cal A}_4$ are satisfied, then
\begin{align*}
G_\varepsilon \left(\vartheta ,\vartheta _0\right)=
\int_{0}^{T}\frac{\left[S\left(\vartheta
    ,t\right)-S\left(\vartheta_0
    ,t\right)\right]^2}{4S\left(\vartheta
    ,t\right)\sigma \left(t\right) }\xi _{t,\varepsilon }^2 \,{\rm
  d}t\left(1+o\left(1\right)\right)+O_p\left(\frac{\varepsilon }{\psi _\varepsilon}\right).
\end{align*}
\end{lemma}
\begin{proof}
We can write
\begin{align*}
G_\varepsilon \left(\vartheta ,\vartheta
_0\right)=-\int_{0}^{T}\frac{M\left(\vartheta ,t\right)-M\left(\vartheta
  _0,t\right)}{ \psi _\varepsilon\sigma \left(t\right)}{\rm d}\bar
W_t+\int_{0}^{T}\frac{\left[M\left(\vartheta ,t\right)-M\left(\vartheta
  _0,t\right)\right]^2}{2\varepsilon \psi_\varepsilon\sigma\left(t\right) ^2}{\rm d}t,
\end{align*}
where
\begin{align*}
M\left(\vartheta ,t\right)-M\left(\vartheta _0,t\right)&=f\left(\vartheta
,t\right)\left[m\left(\vartheta ,t\right)-m\left(\vartheta
  _0,t\right)\right]\\
&\qquad +\left[f\left(\vartheta,t\right)- f\left(\vartheta _0
  ,t\right)\right]m\left(\vartheta _0,t\right).
\end{align*}
To study the difference $m_t\left(\vartheta ,\vartheta
_0\right)=m\left(\vartheta ,t\right)-m\left(\vartheta_0 ,t\right)$, subtract
the equations
\begin{align*}
{\rm d}m\left(\vartheta ,t\right)&=-q_\varepsilon \left(\vartheta,t
\right)m\left(\vartheta ,t\right){\rm d}t +\psi_\varepsilon A\left(\vartheta
,t\right) {\rm d}\bar W_t,\\ {\rm d}m\left(\vartheta_0
,t\right)&=-a\left(\vartheta_0,t \right)m\left(\vartheta_0 ,t\right){\rm d}t
+\psi_\varepsilon A\left(\vartheta_0 ,t\right)\sigma \left(t\right){\rm d}\bar
W_t
\end{align*}
to obtain
\begin{align*}
&{\rm d}m_t\left(\vartheta ,\vartheta _0\right)=-q_\varepsilon
  \left(\vartheta,t \right)m_t\left(\vartheta ,\vartheta _0\right){\rm
    d}t+\psi_\varepsilon\left[ A\left(\vartheta
    ,t\right)-A\left(\vartheta_0 ,t\right)\right]\sigma \left(t\right){\rm
    d}\bar W_t\\
&\qquad -a\left(\vartheta ,t\right)m\left(\vartheta
  _0,t\right){\rm d}t-\frac{\psi _\varepsilon}{\varepsilon } A\left(\vartheta
  ,t\right)\left[f\left(\vartheta ,t\right)-f\left(\vartheta
    _0,t\right)\right]m\left(\vartheta _0,t\right) {\rm d}t.
\end{align*}
The solution of this equation is
\begin{align*}
m_t\left(\vartheta ,\vartheta _0\right)&=\psi_\varepsilon\int_{0}^{t}
e^{-\int_{s}^{t}q_\varepsilon \left(\vartheta ,v\right){\rm d}v}\left[
  A\left(\vartheta ,s\right)-A\left(\vartheta_0 ,s\right)\right]\sigma
\left(s\right){\rm d}\bar W_s\\ &\quad -\frac{\psi_\varepsilon}{\varepsilon
}\int_{0}^{t} e^{-\int_{s}^{t}q_\varepsilon \left(\vartheta ,v\right){\rm d}v}
A\left(\vartheta ,s\right)\left[f\left(\vartheta ,s\right)-f\left(\vartheta
  _0,s\right)\right]m\left(\vartheta _0,s\right){\rm d}s \\ &\quad
-\int_{0}^{t} e^{-\int_{s}^{t}q_\varepsilon \left(\vartheta ,v\right){\rm
    d}v}a\left(\vartheta ,s\right)m\left(\vartheta _0,s\right){\rm d}s.
   \end{align*}
By the same arguments as above,  the following expansions of the
integrals holds,
\begin{align*}
&\psi_\varepsilon\int_{0}^{t} e^{-\int_{s}^{t}q_\varepsilon
  \left(\vartheta ,v\right){\rm d}v}\left[ A\left(\vartheta ,s\right)-A\left(\vartheta_0
  ,s\right)\right]\sigma \left(s\right){\rm d}\bar W_s\\
&\qquad =\sqrt{\varepsilon \psi _\varepsilon}  \frac{\left[
    A\left(\vartheta ,t\right)-A\left(\vartheta_0
  ,t\right)\right]\sigma \left(t\right)}{\sqrt{2A\left(\vartheta
    ,t\right)f\left(\vartheta ,t\right)} }\xi _{t,\varepsilon }
\left(1+o\left(1\right)\right),\\
&\int_{0}^{t} e^{-\int_{s}^{t}q_\varepsilon
  \left(\vartheta ,v\right){\rm d}v}a\left(\vartheta ,s\right)m\left(\vartheta
_0,s\right){\rm d}s\\
&\qquad =\frac{\varepsilon }{\psi _\varepsilon}\frac{a\left(\vartheta ,t\right)m\left(\vartheta
_0,t\right)}{  A\left(\vartheta
    ,t\right)f\left(\vartheta ,t\right)  }\left(1+o_p\left(1\right)\right),\\
&\frac{\psi_\varepsilon}{\varepsilon }\int_{0}^{t} e^{-\int_{s}^{t}q_\varepsilon
  \left(\vartheta ,v\right){\rm d}v} A\left(\vartheta
,s\right)\left[f\left(\vartheta ,s\right)-f\left(\vartheta
  _0,s\right)\right]m\left(\vartheta _0,s\right){\rm d}s\\
&\qquad =\frac{A\left(\vartheta
,t\right)\left[f\left(\vartheta ,t\right)-f\left(\vartheta
  _0,t\right)\right]m\left(\vartheta _0,t\right)}{A\left(\vartheta
,t\right)f\left(\vartheta ,t\right)}+O_p\left(\frac{\varepsilon }{\psi
  _\varepsilon}\right)\\
&\qquad \qquad +\frac{\psi_\varepsilon}{\varepsilon }H_t\left(\vartheta ,\vartheta
_0\right)\int_{0}^{t} e^{-\int_{s}^{t}q_\varepsilon
  \left(\vartheta ,v\right){\rm d}v}\left[m\left(\vartheta _0,s\right)-m\left(\vartheta _0,t\right)
  \right]{\rm d}s .
\end{align*}
Here we denoted $H_t\left(\vartheta ,\vartheta _0\right)=A\left(\vartheta
,t\right)\left[f\left(\vartheta ,t\right)-f\left(\vartheta
  _0,t\right)\right]$. The difference $m\left(\vartheta
_0,s\right)-m\left(\vartheta _0,t\right) $ was already evaluated in
\eqref{2-16}. Therefore we can write
\begin{align*}
m_t\left(\vartheta ,\vartheta _0\right)&= \sqrt{\varepsilon \psi _\varepsilon}  \frac{\left[
    A\left(\vartheta ,t\right)-A\left(\vartheta_0
  ,t\right)\right]\sigma \left(t\right)}{\sqrt{2A\left(\vartheta
    ,t\right)f\left(\vartheta ,t\right)} }\xi _{t,\varepsilon }
\left(1+o_p\left(1\right)\right)\\
&\; -\frac{A\left(\vartheta
,t\right)\left[f\left(\vartheta ,t\right)-f\left(\vartheta
  _0,t\right)\right]m\left(\vartheta _0,t\right)}{A\left(\vartheta
,t\right)f\left(\vartheta ,t\right)}+O_p\left(\frac{\varepsilon }{\psi
  _\varepsilon}\right)\\
&\;  +\sqrt{\varepsilon \psi _\varepsilon} \frac{A\left(\vartheta
,t\right)\left[f\left(\vartheta ,t\right)-f\left(\vartheta
  _0,t\right)\right]A\left(\vartheta_0
  ,t\right) \sigma \left(t\right)}{\sqrt{2} A\left(\vartheta
,t\right)^{3/2}f\left(\vartheta ,t\right)^{3/2}  }\xi _{t,\varepsilon }
\left(1+o_p\left(1\right)\right)\\
&=\frac{\left[f\left(\vartheta _0 ,t\right)-f\left(\vartheta
 ,t\right)\right]m\left(\vartheta _0,t\right)}{f\left(\vartheta
  ,t\right)}+O_p\left(\frac{\varepsilon }{\psi
  _\varepsilon}\right)\\
&\; +\sqrt{\varepsilon \psi _\varepsilon}\frac{\left[A\left(\vartheta
,t\right)f\left(\vartheta ,t\right)-A\left(\vartheta_0
,t\right)f\left(\vartheta_0 ,t\right)\right]\sigma \left(t\right) }{f\left(\vartheta ,t\right)\sqrt{2A\left(\vartheta
,t\right)f\left(\vartheta ,t\right) }}\xi _{t,\varepsilon }
\left(1+o_p\left(1\right)\right).
\end{align*}
Hence
\begin{align*}
&M\left(\vartheta ,t\right)-M\left(\vartheta_0 ,t\right)=O_p\left(\frac{\varepsilon }{\psi
  _\varepsilon}\right)\\
&\qquad +\sqrt{\varepsilon \psi
  _\varepsilon}\frac{\left[b\left(\vartheta
    ,t\right)f\left(\vartheta ,t\right)-b\left(\vartheta_0
    ,t\right)f\left(\vartheta_0 ,t\right)\right]\sqrt{\sigma \left(t\right)}
}{\sqrt{2b\left(\vartheta ,t\right)f\left(\vartheta
    ,t\right) }}\xi _{t,\varepsilon } \left(1+o_p\left(1\right)\right).
\end{align*}
For the stochastic integral we obtain the relation
\begin{align*}
&\int_{0}^{T}\frac{M\left(\vartheta ,t\right)-M\left(\vartheta_0
  ,t\right)}{\psi _\varepsilon \sigma \left(t\right)} {\rm d}\bar
W_t\\
&\qquad =\sqrt{\frac{\varepsilon }{\psi _\varepsilon}}
\int_{0}^{T}\frac{S\left(\vartheta ,t\right)-S\left(\vartheta
  _0,t\right)}{\sqrt{2S\left(\vartheta ,t\right)\sigma \left(t\right)} }
 \xi _{t,\varepsilon }{\rm d}\bar
W_t\left(1+o_p\left(1\right)\right)\longrightarrow 0.
\end{align*}
For the  ordinary integral this gives us the limit
\begin{align}
\label{2-18}
&\int_{0}^{T}\frac{\left[M\left(\vartheta ,t\right)-M\left(\vartheta_0
  ,t\right)\right]^2}{2\varepsilon \psi _\varepsilon \sigma
    \left(t\right)^2}{\rm d}t\nonumber\\
&\qquad \qquad =\int_{0}^{T}\frac{\left[S\left(\vartheta ,t\right)-S\left(\vartheta
  _0,t\right)\right]^2}{{4S\left(\vartheta ,t\right)\sigma \left(t\right)} }
  \xi _{t,\varepsilon }^2{\rm d}t \left(1+o\left(1\right)\right)+O_p\left(\frac{\varepsilon }{\psi
  _\varepsilon}\right)\nonumber\\
&\qquad \qquad \longrightarrow \int_{0}^{T}\frac{\left[S\left(\vartheta ,t\right)-S\left(\vartheta
  _0,t\right)\right]^2}{{4S\left(\vartheta ,t\right)\sigma \left(t\right)} }
  {\rm d}t=G\left(\vartheta ,\vartheta _0\right).
\end{align}

\end{proof}

\begin{lemma}
\label{L6} Assume that conditions ${\cal A},{\cal B}$ are satisfied, then
for some  constant $\kappa >0$ and any $N>0$  there exists a constant $C_N >0$ such that
\begin{align}
\label{2-19}
\Pb_{\vartheta _0}\left\{Z_\varepsilon \left(u\right)\geq e^{-\kappa
  u^2}\right\}\leq \frac{C_N}{\left|u\right|^N} .
\end{align}
\end{lemma}
\begin{proof}
Take a sufficiently small $\delta  $ such that for $\left|h\right|\leq \delta $
\begin{align*}
G\left(\vartheta _0+h,\vartheta _0\right)=\frac12 h^2{\rm I}\left(\vartheta
_0\right)\left(1+o\left(1\right)\right)\geq  \frac{{\rm I}\left(\vartheta
_0\right)}{4}\;h^2.
\end{align*}
Let us denote $g\left(\vartheta _0,\delta \right)=\inf_{\left|\vartheta -\vartheta _0\right|>\delta
}G\left(\vartheta ,\vartheta _0\right) $ and recall that by condition ${\cal
  B}_2$ $g\left(\vartheta _0,\delta \right)>0 $. Then for
$\left|h\right|>\delta $ we have
\begin{align*}
G\left(\vartheta _0+h,\vartheta _0\right)\geq g\left(\vartheta _0,\delta \right)\geq
 \frac{g\left(\vartheta _0,\delta \right)}{\left(\beta -\alpha \right)^2}\;h^2.
\end{align*}
Combining these two estimates we obtain
\begin{align}
G\left(\vartheta _0+h,\vartheta _0\right)\geq \kappa _* \;h^2,\qquad
\kappa_*=\left(\frac{g\left(\vartheta _0,\delta \right)}{\left(\beta -\alpha \right)^2} \wedge
  \frac{{\rm I}\left(\vartheta _0\right)}{4}\right).
\label{2-20}
\end{align}
Denote $\Delta M_t=M\left(\vartheta_u ,t \right)-M\left(\vartheta
_0 ,t \right)$ and write
\begin{align}
&\Pb_{\vartheta _0}\left\{\ln Z_\varepsilon \left(u\right)\geq {-\kappa
    u^2}\right\}\nonumber\\
&\qquad =\Pb_{\vartheta _0}\left\{\int_{0}^{T}\frac{\Delta
    M_t}{2\varepsilon \sigma \left(t\right)}{\rm d}\bar W_t-\int_{0}^{T}\frac{\Delta
    M_t^2}{4\varepsilon^2 \sigma \left(t\right)^2}{\rm d}t\geq {-\frac{\kappa
    }{2}    u^2}\right\} \nonumber\\
&\qquad \leq \Pb_{\vartheta _0}\left\{\int_{0}^{T}\frac{\Delta
    M_t}{2\varepsilon \sigma \left(t\right)}{\rm d}\bar W_t-\int_{0}^{T}\frac{\Delta
    M_t^2}{8\varepsilon^2 \sigma \left(t\right)^2}{\rm d}t\geq
  {\frac{\kappa}{2}    u^2}\right\}\nonumber\\
&\qquad \qquad \qquad \qquad +\Pb_{\vartheta _0}\left\{-\int_{0}^{T}\frac{\Delta
    M_t^2}{8\varepsilon^2 \sigma \left(t\right)^2}{\rm d}t\geq {-\kappa
        u^2}\right\}\nonumber\\
&\qquad \leq e^{-\frac{\kappa }{2}u^2} +\Pb_{\vartheta
    _0}\left\{\int_{0}^{T}\frac{\Delta
    M_t^2}{\varepsilon^2 \sigma \left(t\right)^2}{\rm d}t\leq
  {8\kappa        u^2}\right\}.
\label{2-21}
\end{align}
Here we used the equality
\begin{align*}
\Ex_{\vartheta _0}\exp\left\{ \int_{0}^{T}\frac{\Delta
    M_t}{2\varepsilon \sigma \left(t\right)}{\rm d}\bar W_t-\int_{0}^{T}\frac{\Delta
    M_t^2}{8\varepsilon^2 \sigma \left(t\right)^2}{\rm d}t  \right\}=1.
\end{align*}
To estimate the latter probability in \eqref{2-21} we consider separately two
cases $\AA=  \left(u:
\left|u\right|\leq \varphi _\varepsilon ^{1/2}\right)$ (local) and $\AA^c=\left(u:
\left|u\right|>\varphi _\varepsilon ^{1/2}\right)$ (global). Locality is with respect to the values of $u$
for which   $\left|\vartheta_u
-\vartheta _0\right|\leq \varphi _\varepsilon ^{1/2}$.

Let $u\in\AA$. The proof of \eqref{2-14} shows that this
representation is valid for the values $u\in\AA $. Moreover the residuals
$o\left(1\right)$ and $O\left(\frac{\varepsilon }{\psi \left(\varepsilon
  \right)}\right)$ have bounded polynomial moments of all orders.
Hence
\begin{align*}
&\frac{M\left(\vartheta _0+\varphi _\varepsilon u,t\right)-M\left(\vartheta
    _0,t\right) }{\varepsilon \sigma \left(t\right)} =\frac{\varphi
    _\varepsilon u}{\varepsilon } \frac{\dot M\left(\vartheta
    _0,t\right)}{\sigma
    \left(t\right)}\left(1+o_p\left(1\right)\right)\\
 &\qquad=\frac{
    u}{\sqrt{\varepsilon\psi _\varepsilon} } \frac{\dot
    M\left(\vartheta _0,t\right)}{\sigma
    \left(t\right)}\left(1+o_p\left(1\right)\right)\\
&\qquad = \frac{u \;\dot
    S\left(\vartheta_0,t\right) }{\sqrt{2S\left(\vartheta_0 ,t\right)\sigma
      \left(t\right)}} \; {}{ }\; \xi _{t,\varepsilon}
  \left(1+o_p\left(1\right)\right)+u O_p\left(\left(\frac{\varepsilon }{\psi
    _\varepsilon^3 }\right)^{1/2}\right).
\end{align*}
 Therefore we can write
\begin{align*}
&\int_{0}^{T}\frac{\Delta
    M_t^2}{\varepsilon^2 \sigma \left(t\right)^2}{\rm d}t=u^2
\int_{0}^{T}\frac{\dot S\left(\vartheta_0,t\right)^2
}{2S\left(\vartheta_0,t\right)\sigma \left(t\right) }\xi _{t,\varepsilon
}^2{\rm d}t \left(1+o_p\left(1\right)\right)+u^2 O_p\left(\frac{\varepsilon }{\psi
    _\varepsilon^3 }\right)\\
&\quad =u^2{\rm I}\left(\vartheta _0\right)+u^2
\int_{0}^{T}\frac{\dot S\left(\vartheta_0,t\right)^2
}{2S\left(\vartheta_0,t\right)\sigma \left(t\right) }\left(\xi _{t,\varepsilon
}^2-1 \right){\rm d}t \left(1+o\left(1\right)\right)+u^2 O_p\left(\frac{\varepsilon }{\psi
    _\varepsilon^3 }\right)\\
&\quad =u^2{\rm I}\left(\vartheta _0\right)+u^2 \sqrt{\frac{\varepsilon }{\psi
    _\varepsilon}} N_\varepsilon \left(1+o\left(1\right)\right)+u^2 \frac{\varepsilon }{\psi
    _\varepsilon^3 }Q_\varepsilon
\end{align*}
with the obvious notations. Further, let us denote $\hat\kappa =\inf_{\vartheta
  \in\Theta }{\rm I}\left(\vartheta \right)>0 $ and introduce the sets
\begin{align*}
\NN_\varepsilon=\left\{ \sqrt{\frac{\varepsilon }{\psi \left(\varepsilon
    \right)}} \left|N_\varepsilon \right|\left(1+o\left(1\right)\right)\leq
\frac{\hat\kappa}{4} \right\} ,\qquad \qquad \QQ_\varepsilon=\left\{
\frac{\varepsilon }{\psi_\varepsilon^3 } \left|Q_\varepsilon
\right)\leq \frac{\hat\kappa}{4} \right\}.
\end{align*}
 Then we can write
\begin{align*}
&\Pb_{\vartheta
    _0}\left\{\int_{0}^{T}\frac{\Delta
    M_t^2}{\varepsilon^2 \sigma \left(t\right)^2}{\rm d}t \leq
  {8\kappa        u^2}\right\}\\
&\qquad \qquad\leq \Pb_{\vartheta
    _0}\left\{\int_{0}^{T}\frac{\Delta
    M_t^2}{\varepsilon^2 \sigma \left(t\right)^2}{\rm d}t\leq
  {8\kappa        u^2}, \NN_\varepsilon ,\QQ_\varepsilon \right\}+\Pb_{\vartheta
    _0}\left(\NN_\varepsilon ^c \right)+\Pb_{\vartheta
    _0}\left(\QQ_\varepsilon ^c \right).
\end{align*}
If we let $\kappa
=\hat\kappa/32$ then for $\left|u\right|>0$ the first probability satisfies
\begin{align*}
 &\Pb_{\vartheta
    _0}\left\{\int_{0}^{T}\frac{\Delta
    M_t^2}{\varepsilon^2 \sigma \left(t\right)^2}{\rm d}t\leq
  {8\kappa        u^2}, \NN_\varepsilon ,\QQ_\varepsilon \right\}\\
&\qquad  \quad\qquad  \quad\leq \Pb_{\vartheta
    _0}\left\{u^2{\rm I}\left(\vartheta _0\right)-\frac{\hat\kappa}{2}u^2    \leq
  {8\kappa        u^2}, \NN_\varepsilon ,\QQ_\varepsilon \right\}\\
&\qquad  \quad\qquad  \quad \leq \Pb_{\vartheta
    _0}\left\{\frac{\hat\kappa}{2}u^2    \leq
  {8\kappa        u^2}, \NN_\varepsilon ,\QQ_\varepsilon \right\}=0.
\end{align*}
As the moments of $N_\varepsilon $ and $Q_\varepsilon $ are bounded we can
write, for any $n>0$,
\begin{align*}
\Pb_{\vartheta _0}\left\{ \NN_\varepsilon ^c\right\}&\leq C\,\left(\frac{\varepsilon }{\psi
  _\varepsilon}\right)^{\frac{n}{2}}=C\varepsilon
^{\frac{n\left(1-\delta \right)}{2}},\\
 \Pb_{\vartheta _0}\left\{ \QQ_\varepsilon^c \right\}&\leq C\,\left(\frac{\varepsilon }{\psi
  _\varepsilon^3}\right)^{n}=C\varepsilon ^{n\left(1-3\delta \right)}.
\end{align*}
Let   $\delta _*=\frac{\left(1-\delta \right)}{2} \wedge
\left(1-3\delta\right) $, then
\begin{align*}
&\Pb_{\vartheta _0}\left\{\int_{0}^{T}\frac{\Delta M_t^2}{\varepsilon^2 \sigma
    \left(t\right)^2}{\rm d}t \leq {8\kappa u^2}\right\}\leq \Pb_{\vartheta
    _0}\left\{ \NN_\varepsilon ^c\right\}+\Pb_{\vartheta _0}\left\{
  \QQ_\varepsilon ^c\right\} \leq C\varepsilon ^{n\delta _*}.
\end{align*}
Recall that
\begin{align*}
\vartheta _0+\varphi _\varepsilon u\in \left(\alpha ,\beta \right), \qquad
\left|u\right|\leq \frac{\beta -\alpha }{\varphi _\varepsilon }=\varepsilon
^{-\frac{1-\delta }{2}}\left(\beta -\alpha \right),\qquad \varepsilon <
\frac{\left(\beta -\alpha \right)^{\frac{2}{1-\delta }}}{\left|u\right|^{\frac{2}{1-\delta }}}.
\end{align*}
Hence if  for any $N>0$ we  take     $n=\frac{N\left(1-\delta \right)}{2\delta
  _*} $ then
\begin{align*}
\Pb_{\vartheta _0}\left\{\int_{0}^{T}\frac{\Delta M_t^2}{\varepsilon^2 \sigma
    \left(t\right)^2}{\rm d}t \leq {8\kappa u^2}\right\}\leq
\frac{C}{\left|u\right|^{\frac{2n\delta _*}{1-\delta }}  }\leq \frac{C}{\left|u\right|^N}.
\end{align*}
Thus we obtained the estimate \eqref{2-19} for $u\in \AA$.

Suppose that $u\in \AA^c$. The relations \eqref{2-18} and \eqref{2-20}  allow us to write
\begin{align*}
&\int_{0}^{T}\frac{\Delta M_t^2}{2\varepsilon^2 \sigma \left(t\right)^2}\;{\rm
  d}t=\frac{\psi _\varepsilon}{\varepsilon }\;\frac{1
}{\varepsilon \psi _\varepsilon}\int_{0}^{T}\frac{\Delta
  M_t^2}{2\sigma \left(t\right)^2}\;{\rm d}t\\
&\qquad =\frac{\psi_\varepsilon}{\varepsilon }\;\int_{0}^{T}\frac{\left[S\left(\vartheta
    ,t\right)-S\left(\vartheta_0 ,t\right)\right]^2}{4S\left(\vartheta
  ,t\right) \sigma \left(t\right)}\xi _{t,\varepsilon }^2\;{\rm
    d}t\left(1+o_p\left(1\right)\right)+O_p\left(1\right)\\
&\qquad =\frac{\psi_\varepsilon}{\varepsilon }\;G\left(\vartheta_u,\vartheta
  _0\right) +\sqrt{\frac{\psi_\varepsilon}{\varepsilon }}\;R_\varepsilon +O\left(1\right) \geq \kappa _*
  u^2+\sqrt{\frac{\psi_\varepsilon}{\varepsilon }}\;R_\varepsilon +O\left(1\right),
\end{align*}
where we denoted
\begin{align*}
R_\varepsilon=\sqrt{\frac{\psi_\varepsilon}{\varepsilon }}\int_{0}^{T}\frac{\left[S\left(\vartheta
    ,t\right)-S\left(\vartheta_0 ,t\right)\right]^2}{4S\left(\vartheta
  ,t\right) \sigma \left(t\right)}\left[\xi _{t,\varepsilon }^2-1\right]\;{\rm
    d}t\left(1+o_p\left(1\right)\right).
\end{align*}
Therefore if we take $\kappa =\kappa _*/16$ and recall that for $u\in \AA^c$
we have $\left|u\right|>\sqrt{\frac{\psi \left(\varepsilon
    \right)}{\varepsilon }} $ then we can write
\begin{align*}
&\Pb_{\vartheta _0}\left\{\int_{0}^{T}\frac{\Delta M_t^2}{\varepsilon^2 \sigma
    \left(t\right)^2}{\rm d}t \leq {8\kappa u^2}\right\}\\ &\qquad \qquad \leq
  \Pb_{\vartheta _0}\left\{-\sqrt{\frac{\psi \left(\varepsilon
      \right)}{\varepsilon }}\;\left|R_\varepsilon\right| +O_p\left(1\right)\leq
  -{\left(\kappa _*-8\kappa\right) u^2}\right\}\\ &\qquad \qquad \leq
  \Pb_{\vartheta _0}\left\{
  \left|u\right|\;\left|R_\varepsilon\right|+O_p\left(1\right) \geq \frac{\kappa
    _*}{2} u^2\right\}\leq \frac{C}{\left|u\right|^N}.
\end{align*}

\end{proof}

\begin{lemma}
\label{L7}
Assume that  conditions ${\cal A}$ and ${\cal B}_1$ are satisfied, then
\begin{align}
\label{2-22}
\Ex_{\vartheta _0}\left[Z_\varepsilon \left(u_1\right)^{1/2}-Z_\varepsilon
 \left(u_2\right)^{1/2} \right]^2\leq C\left|u_2-u_1\right|^2
\end{align}

\end{lemma}
\begin{proof} As usually in such situations (see, e.g., \cite{Kut94}) we write
\begin{align*}
\Ex_{\vartheta _0}\left[Z_\varepsilon \left(u_1\right)^{1/2}-Z_\varepsilon
  \left(u_2\right)^{1/2} \right]^2= 2-2 \Ex_{\vartheta
  _{u_1}}\left(\frac{Z_\varepsilon \left(u_2\right)}{Z_\varepsilon
  \left(u_1\right)}\right)^{1/2.}= 2-2 \Ex_{\vartheta _{u_1}}V_T,
\end{align*}
where we denoted $\vartheta _{u_1}=\vartheta _0+\varphi _\varepsilon u_1$, and  introduce the process
\begin{align*}
V_t=\exp\left\{\int_{0}^{t}\frac{\Delta M_s}{2\varepsilon ^2\sigma
  \left(s\right)^2}{\rm d}X_s-\int_{0}^{t}\frac{\left[M\left(\vartheta
  _{u_2},s\right)^2-M\left(\vartheta _{u_1},s\right)^2\right]}{4\varepsilon ^2\sigma
  \left(s\right)^2}{\rm d}s\right\}.
\end{align*}
Here $0\leq t\leq T $ and $\Delta M_s=M\left(\vartheta _0+\varphi _\varepsilon
{u_2},s\right)-M\left(\vartheta _0+\varphi _\varepsilon {u_1},s\right)$.  This
process with $\Pb_{\vartheta _{u_1}}$-probability 1 has stochastic differential
\begin{align*}
{\rm d}V_t=   -\frac{\left(\Delta M_t\right)^2}{8\varepsilon ^2\sigma
  \left(t\right)^2}V_t{\rm d}t+ \frac{\Delta M_t}{2\varepsilon \sigma
  \left(t\right)}V_t{\rm d}\bar W_t,\qquad V_0=1.
\end{align*}
Therefore
\begin{align*}
2-2 \Ex_{\vartheta _{u_1}}V_T&=\int_{0}^{T}\Ex_{\vartheta _{u_1}}V_t
\frac{\left(\Delta M_t\right)^2}{8\varepsilon ^2\sigma \left(t\right)^2}{\rm
  d}t\\
&\leq \frac{1}{4\varepsilon ^2}\int_{0}^{T}
\frac{\Ex_{\vartheta _{u_2}}\left(\Delta M_t\right)^2}{\sigma \left(t\right)^2}{\rm
  d}t+\frac{1}{4\varepsilon ^2}\int_{0}^{T}
\frac{\Ex_{\vartheta _{u_1}}\left(\Delta M_t\right)^2}{\sigma \left(t\right)^2}{\rm
  d}t.
\end{align*}
Here we used the inequality $V_t \left(\Delta M_t\right)^2\leq 2V_t^2
  \left(\Delta M_t\right)^2 +2 \left(\Delta M_t\right)^2$ and changed the
  measure $\Ex_{\vartheta _{u_1}}V_t^2=\Ex_{\vartheta _{u_2}} $.
Now the bound \eqref{2-22} follows from the representation \eqref{2-14}, where
$\vartheta _0$ is replaced with  $\vartheta _{u_1}$ and $\vartheta _{u_2}$ respectively.

\end{proof}

It can be shown that the convergence in Lemma \ref{L4} is uniform
on the compacts of $\Theta $ and the constants in the Lemmas  \ref{L6} and
\ref{L7} can be chosen independent on $\vartheta _0$.

The properties of the normalized likelihood ratio $Z_\varepsilon \left(\cdot
\right)$ established in Lemmas \ref{L4}, \ref{L6} and  \ref{L7}
verify the sufficient conditions $N1-N4$  of  Theorems  3.1.1 and
3.2.1 in \cite{IH81}, which, in turn, imply the properties of the MLE and BE claimed in
 Theorem \ref{T1}.

Since the  convergence  of moments is
uniform on compacts of $\Theta $,  we can prove the asymptotic efficiency
of estimators as follows. The uniform convergence of moments for the MLE imply
\begin{align*}
\lim_{\varepsilon \rightarrow 0}\sup_{\left|\vartheta -\vartheta _0\right|\leq
  \nu }\frac{\psi _\varepsilon}{ \varepsilon }\Ex_{\vartheta
}\left[\hat\vartheta _\varepsilon-\vartheta \right]^2=\sup_{\left|\vartheta
  -\vartheta _0\right|\leq \nu }{\rm I}\left(\vartheta
\right)^{-1}\xrightarrow{\nu \rightarrow 0} {\rm I}\left(\vartheta_0 \right)^{-1}.
\end{align*}
The same convergence holds for BE $\tilde \vartheta _\varepsilon $. For more
general loss functions see Theorem 3.1.3 in \cite{IH81}.

 \end{proof}

\section{Discussions}

 Suppose that $f\left(\vartheta ,t\right)=\vartheta
f\left(t\right),b\left(\vartheta ,t\right)=\vartheta ^{-1}b\left(t\right)$ and
$a\left(\vartheta ,t\right)$ depend  on $\vartheta $. Then the function
$S\left(\vartheta ,t\right)=f\left(t\right)b\left(t\right)$ does not depend on
$\vartheta $ and the conditions ${\cal B}$ fail. Now the
existence of the consistent estimator depends on the value of $y_0$. If
$y_0\not=0$ then the rate of convergence of the estimators is different. Let
us construct a consistent and asymptotically normal estimator in this
situation. Introduce the notations
\begin{align*}
H\left(\vartheta ,t\right)&=y_0\int_{0}^{t}\exp
\left(\int_{0}^{s}a\left(\vartheta ,v\right){\rm d}v\right)f\left(s\right)
     {\rm d}s ,\qquad \eta \left(t\right)=\int_{0}^{t}\sigma
     \left(s\right){\rm d}W_s,\\
\pi \left(\vartheta
     ,t\right)&=\int_{0}^{t}\int_{0}^{s}\exp\left\{\int_{r}^{s}a\left(\vartheta
     ,v\right)\,{\rm d}v\right\}b\left(r\right)\,{\rm
       d}V_r\,f\left(s\right)\,{\rm d}s .
\end{align*}
Then the observed process can be written as follows
\begin{align*}
X_t=H\left(\vartheta_0 ,t\right)+\psi _\varepsilon \pi \left(\vartheta_0
,t\right)+\varepsilon \eta\left(t\right),\qquad 0\leq t\leq T,
\end{align*}
where $\pi \left(\cdot \right)$ and $\eta \left(\cdot \right)$ are independent
Gaussian processes. The {\it minimum distance estimator } (MDE) $\vartheta
_\varepsilon ^*$ is the solution of equation
\begin{align*}
\int_{0}^{T}\left[X_t-H\left(\vartheta _\varepsilon ^*,t\right)\right]^2{\rm
  d}t=\inf_{\vartheta \in\Theta }\int_{0}^{T}\left[X_t-H\left(\vartheta
  ,t\right)\right]^2{\rm d}t.
\end{align*}
The identifiability condition is
\begin{align}
\label{3-1}
\inf_{\left|\vartheta -\vartheta _0\right|>\nu }\int_{0}^{T}\left[H\left(\vartheta
  ,t\right)-H\left(\vartheta_0
  ,t\right)\right]^2{\rm d}t>0, \quad  \forall \nu >0,
\end{align}
If this condition is satisfied then the MDE is consistent (see \cite{Kut94}). Moreover, it can be
shown that
\begin{align*}
\vartheta _\varepsilon ^*&=\vartheta _0+\psi _\varepsilon \int_{0}^{T}\pi
\left(\vartheta _0,t\right) D\left(\vartheta_0 ,t\right){\rm
  d}t\left(1+o\left(1\right)\right)\\
&\qquad  +\varepsilon \int_{0}^{T}\eta
\left(t\right) D\left(\vartheta_0 ,t\right){\rm
  d}t\left(1+o\left(1\right)\right),
\qquad  D\left(\vartheta_0 ,t\right)
=\frac{\dot H\left(\vartheta _0,t\right)}{\int_{0}^{T}\dot H\left(\vartheta _0,t\right)^2{\rm d}t}.
\end{align*}
Therefore if we denote $\phi_\varepsilon =\max(\varepsilon ,\psi _\varepsilon)
$, then for any $\varepsilon \rightarrow 0$ and $\psi _\varepsilon \rightarrow
0$, the asymptotic normality of $\vartheta _\varepsilon ^*$ holds,
\begin{align}
\label{3-2}
\frac{\vartheta _\varepsilon ^*-\vartheta _0}{\phi_\varepsilon
}\Longrightarrow {\cal N}\left(0, d\left(\vartheta _0\right)^2\right)
\end{align}
with the corresponding limit variance $d\left(\vartheta _0\right)^2$.

Let us consider the possibility of the adaptive filtration for this model of
observations. As mentioned in the Introduction, finding the MLE and the BE for
the partially observed linear system \eqref{01}, \eqref{02} is computationally
inefficient, since it requires solving the filtering equations \eqref{2-3},
\eqref{2-4} for all $\vartheta \in \Theta $. The for the MLE we have to solve
the maximization problem \eqref{04}. Instead we can use a much simpler
algorithm, based on the multi-step approach, recently developed in
\cite{Kut17},\cite{KhK18}, \cite{KZ20}. Let us consider such construction of
preliminary estimator  in
the case of observations \eqref{01} omitting the technical details. Fix a
small $\tau \in \left(0,T\right)$ and let $\hat \vartheta _{\tau ,\varepsilon }$
be the MLE ($y_0=0$)  and  $ \vartheta _{\tau ,\varepsilon }^*$ be the MDE ($y_0\not=0$)
described above based on the observations $X^\tau =\left(X_t,t\in\left[0,\tau
  \right]\right)$. Assume that the corresponding identifiability and regularity
conditions are fulfilled then the both estimators are 
consistent and asymptotically normal. Below we use notation $ \vartheta _{\tau
  ,\varepsilon }^*$ for preliminary estimator assuming that it can be the MLE
too.   Let $\psi _\varepsilon
=\varepsilon^\delta , $ $\delta \in \left(\frac{1}{5}, \frac{1}{3}\right)$ and
define the one-step MLE-process $\vartheta _{t,\varepsilon} ^\star,\tau <t\leq
T $, where 
\begin{align}
\label{3-3}
\vartheta _{t,\varepsilon} ^\star=\vartheta _{\tau ,\varepsilon }^*+{\rm I}_\tau^t
\left(\vartheta _{\tau ,\varepsilon }^*\right)^{-1}\int_{\tau }^{t}\frac{\dot
  M\left(\vartheta _{\tau ,\varepsilon }^*,s\right)}{\varepsilon \psi _\varepsilon\sigma
  \left(s\right)^2}\left[{\rm d}X_s-M\left(\vartheta _{\tau ,\varepsilon
  }^*,s\right){\rm d}s\right].
\end{align}
The Fisher information here
\begin{align*}
{\rm I}_\tau ^t\left(\vartheta\right)=\int_{\tau }^{t}\frac{\dot
  S\left(\vartheta ,s\right)^2}{2S\left(\vartheta ,s\right)\sigma
  \left(s\right)}{\rm d}s 
\end{align*}
is supposed to be positive for all $t\in (\tau ,T]$. We have to explain how
  calculate $\dot
  M\left(\vartheta _{\tau ,\varepsilon }^*,s\right) $ and $
  M\left(\vartheta _{\tau ,\varepsilon }^*,s\right)$. Recall that by
  \eqref{2-7} ($y_0=0$)  we have
\begin{align*}
m\left(\vartheta ,t\right)&=\frac{\psi _\varepsilon }{\varepsilon
}e^{-\int_{0}^{t}q_\varepsilon \left(\vartheta ,v\right){\rm
    d}v}\int_{0}^{t}e^{\int_{0}^{s}q_\varepsilon \left(\vartheta ,v\right){\rm
    d}v}A\left(\vartheta ,s\right){\rm d}X_s\\
& =h \left(\vartheta
,t\right) \int_{0}^{t}H \left(\vartheta ,s\right){\rm d}X_s 
\end{align*}
with obvious notation. The stochastic integral can be replaced using the
following relation
\begin{align*}
\int_{0}^{t}H \left(\vartheta ,s\right){\rm d}X_s=H \left(\vartheta ,t\right)
X_t-\int_{0}^{t}H'_s \left(\vartheta ,s\right)X_s{\rm d}s. 
\end{align*}
The right hand side of this equality we denote as $Z\left(\vartheta,t
,X^t\right) $ and put
\begin{align*}
m\left(\vartheta _{\tau ,\varepsilon }^* ,s\right)=h \left(\vartheta _{\tau ,\varepsilon }^*
,s\right)Z\left(\vartheta _{\tau ,\varepsilon }^*,s
,X^s\right). 
\end{align*}
For $\dot
  M\left(\vartheta _{\tau ,\varepsilon }^*,s\right) $ can be obtained the
  similar expression.

Following the usual calculations in such situation (see, e.g.,
\cite{KZ14},\cite{Kut17}) we obtain the relations
\begin{align*}
&\sqrt{\frac{\psi _\varepsilon }{\varepsilon }}\left(\vartheta _{t,\varepsilon}
^\star-\vartheta _0\right)= \sqrt{\frac{\psi _\varepsilon }{\varepsilon }}
\left(\vartheta _{\tau ,\varepsilon }^* -\vartheta _0 \right)+\frac{1}{{\rm I}_\tau^t
\left(\vartheta _{\tau ,\varepsilon }^*\right)}\int_{\tau }^{t}\frac{\dot
  M\left(\vartheta _{\tau ,\varepsilon }^*,s\right)}{\sqrt{\varepsilon \psi _\varepsilon}\sigma
  \left(s\right)}{\rm d}\bar W_s\\
&\quad\qquad  +\sqrt{\frac{\psi _\varepsilon }{\varepsilon }}\frac{1}{{\rm I}_\tau^t
\left(\vartheta _{\tau ,\varepsilon }^*\right)}
\int_{\tau }^{t}\frac{\dot
  M\left(\vartheta _{\tau ,\varepsilon }^*,s\right)}{\varepsilon \psi _\varepsilon\sigma
  \left(s\right)^2} \left[M\left(\vartheta _0,s\right)  -M\left(\vartheta _{\tau ,\varepsilon
  }^*,s\right)\right]{\rm d}s\\
&\quad = \sqrt{\frac{\psi _\varepsilon }{\varepsilon }}
\left(\vartheta _{\tau ,\varepsilon }^* -\vartheta _0 \right)+\frac{1}{{\rm I}_\tau^t
\left(\vartheta _0\right)}\int_{\tau }^{t}\frac{\dot
  S\left(\vartheta _0,s\right)}{\sqrt{2S\left(\vartheta _0,s\right)\sigma
  \left(s\right)}}\xi _{s,\varepsilon }{\rm d}\bar W_s\left(1+o\left(1\right)\right)\\
&\quad\qquad  -\sqrt{\frac{\psi _\varepsilon }{\varepsilon }}\frac{\left(\vartheta _{\tau ,\varepsilon }^* -\vartheta _0 \right)}{{\rm I}_\tau^t
\left(\vartheta _0\right)}
\int_{\tau }^{t}\frac{\dot
  S\left(\vartheta _0,s\right)^2}{2S\left(\vartheta _0,s\right)\sigma
  \left(s\right)}\xi _{s,\varepsilon }^2 {\rm d}s\left(1+O\left(\vartheta _{\tau ,\varepsilon }^* -\vartheta _0 \right)\right)\\
&\quad = \frac{1}{{\rm I}_\tau^t
\left(\vartheta _0\right)}\int_{\tau }^{t}\frac{\dot
  S\left(\vartheta _0,s\right)}{\sqrt{2S\left(\vartheta _0,s\right)\sigma
  \left(s\right)}}\xi _{s,\varepsilon }{\rm d}\bar
W_s\left(1+o_p\left(1\right)\right)\\
&\quad\qquad  +\sqrt{\frac{\psi _\varepsilon }{\varepsilon }}
\left(\vartheta _{\tau ,\varepsilon }^* -\vartheta _0 \right)^2O_p\left(1\right)+o_p\left(1\right).
\end{align*}
Since $\delta \in \left(\frac{1}{5},\frac{1}{3}\right)$,
\begin{align*}
\sqrt{\frac{\psi _\varepsilon }{\varepsilon }}
\left(\vartheta _{\tau ,\varepsilon }^* -\vartheta _0
\right)^2=\sqrt{\frac{\psi _\varepsilon }{\varepsilon }} \psi _\varepsilon
^2\;O_p\left(1\right)=\varepsilon ^{\frac{5}{2}\left(\delta
  -\frac{1}{5}\right)}\;O_p\left(1\right)\longrightarrow  0
\end{align*}
and hence
\begin{align*}
\int_{\tau }^{t}\frac{\dot
  S\left(\vartheta _0,s\right)}{\sqrt{2S\left(\vartheta _0,s\right)\sigma
  \left(s\right)}}\xi _{s,\varepsilon }{\rm d}\bar
W_s\Longrightarrow {\cal N}\left(0,{\rm I}_\tau^t
\left(\vartheta _0\right)\right)
\end{align*}
and, consequently,
\begin{align*}
   \sqrt{\frac{\psi _\varepsilon }{\varepsilon }}\left(\vartheta _{t,\varepsilon}
^\star-\vartheta _0\right)         \Longrightarrow {\cal N}\left(0,{\rm I}_\tau^t
\left(\vartheta _0\right)^{-1}\right).
\end{align*}
Thus we constructed estimator which is consistent and asymptotically normal
with good rate. It requires solving the Riccati equation just for one value
$\vartheta = \vartheta _{\tau ,\varepsilon }^*$ and the random functions
$m(\vartheta _{\tau ,\varepsilon }^*,t)$ and $\dot m(\vartheta
_{\tau ,\varepsilon }^*,t)$. Though the presentation here was formal, all calculations can be made precise using the
technique developed in
\cite{KZ14}, \cite{KhK18}, \cite{Kut16}, \cite{Kut17}, \cite{KZ20}.

The adaptive filtration can be realized with the help of the equations
\eqref{3-3} and
\begin{align*}
{\rm d}\hat m_t=-a\left(\vartheta _{t,\varepsilon}
^\star,t\right)\hat m_t{\rm d}t+\frac{\psi _\varepsilon }{\varepsilon }\frac{b\left(\vartheta _{t,\varepsilon}
^\star,t\right)}{\sigma \left(t\right)}\left[{\rm d}X_t-f\left(\vartheta _{t,\varepsilon}
^\star,t\right)\hat m_t{\rm d}t\right],\; \tau <t\leq T,
\end{align*}
subject to initial value $\hat m_\tau =m(\vartheta _{\tau ,\varepsilon
}^*,\tau ) $. Moreover, as it was shown in the mentioned above works in
similar situations this
estimation of $m\left(\vartheta_0,t\right)$ can have some properties of
optimality. 

\bigskip

 Consider the observations model
\eqref{01},\eqref{02} and assume that conditions ${\cal A}, {\cal B}$ hold. As
$y_0=0$  the processes $X^T,Y^T$ converge with probability  1 to $0$,
\begin{align*}
\sup_{0\leq t\leq T}\left|X_t\right|\longrightarrow 0,\qquad \quad \sup_{0\leq
  t\leq T}\left|Y_t\right|\longrightarrow 0.
\end{align*}
This means that the limit observations are $X_t\equiv 0$, but nevertheless the
MLE and BE still have all the properties, claimed in  Theorem \ref{T1}. On the other hand,
this shows the essential difference between the observation model  \eqref{01},\eqref{02} with
 $\psi _\varepsilon =\varepsilon $ studied in \cite{Kut94} and the present one. Unlike in out case,
there we have a
deterministic dynamical system \eqref{06} (limit model, $y_0\not=0$) perturbed
by small noise.

The result of this paper is in a sense surprising. We see that the
error of estimation decreases if noise intensity $\psi _\varepsilon $ in the state equation
increases in some region ($\psi _\varepsilon =\varepsilon ^\delta ,0<\delta
<\frac{1}{3}$). This can be explained heuristically as follows. Suppose that the
conditions ${\cal A}, {\cal B}$ hold and
$y_0=0$. Then
\begin{align*}
Y_t=\psi _\varepsilon \int_{0}^{t}e^{-\int_{s}^{t}a\left(\vartheta _0,v\right){\rm
    d}v}b\left(\vartheta _0,s\right){\rm d}V_s =\psi _\varepsilon \,\hat Y_t\left(\vartheta _0\right),
\end{align*}
where $\hat Y_t\left(\vartheta _0\right)$ is defined by the latter equality.
The observed process can be rewritten as
\begin{align}
\label{3-5}
{\rm d}X_t=\psi _\varepsilon f\left(\vartheta_0
,t\right)\,\hat Y_t\left(\vartheta _0\right) \;{\rm d}t+\varepsilon \sigma
\left(t\right){\rm d}W_t.
\end{align}
Here $\varepsilon \rightarrow 0$ much faster than $\psi _\varepsilon
\rightarrow 0$. Suppose that  $\varepsilon =0$ holds already but $\psi
_\varepsilon >0$. Then the observations $x_t=\frac{{\rm d}X_t}{{\rm d}t},
t\in\left[0,T\right]$ are
$x_t=\psi _\varepsilon f\left(\vartheta _0,t\right)\hat Y_t\left(\vartheta
_0\right),$ $t\in\left[0,T\right]$ and the process
$\tilde x_t=\psi _\varepsilon ^{-1}x_t=f\left(\vartheta _0,t\right)\hat Y_t\left(\vartheta _0\right)$
has the differential
\begin{align*}
{\rm d}\tilde x_t= \left[f'\left(\vartheta
  _0,t\right)-f\left(\vartheta _0,t\right)a\left(\vartheta
  _0,t\right)\right]\hat Y_t\left(\vartheta _0\right){\rm d}t +
S\left(\vartheta _0,t\right){\rm d}V_t,\qquad  \tilde x_0=0,
\end{align*}
where as before $S\left(\vartheta _0,t\right)=f\left(\vartheta
_0,t\right)b\left(\vartheta _0,t\right)$. The It\^o formula for $\tilde x_t^2$
gives the equation
\begin{align*}
{\rm d}\tilde x_t^2=2\tilde x_t{\rm d}\tilde x_t+S\left(\vartheta
_0,t\right)^2{\rm d}t,\qquad  \tilde x_0=0.
\end{align*}
Therefore we can write
\begin{align*}
\int_{0}^{T}S\left(\vartheta _0,t\right)^2{\rm d}t=\tilde x_T^2- 2
\int_{0}^{T}\tilde x_t{\rm d}\tilde x_t .
\end{align*}
This equation (under identifiability condition) allows us to find $\vartheta
_0$ precisely. For example, let $f\left(\vartheta
,t\right)=f\left(t\right)$ and $b\left(\vartheta ,t\right)=\sqrt{\vartheta
}b\left(t\right)$, then
\begin{align*}
\vartheta _0=\left(\int_{0}^{T}f\left(t\right)^2b\left(t\right)^2{\rm
  d}t\right)^{-1} \left[\tilde x_T^2- 2
\int_{0}^{T}\tilde x_t{\rm d}\tilde x_t  \right].
\end{align*}
We see that if $\varepsilon =0$, then the value $\vartheta _0$ by observations
$x_t,t\in\left[0,T\right]$ can be calculated without error.

Note that in observations \eqref{3-5} the function $\psi _\varepsilon $ plays
the role of amplitude of the signal $\psi _\varepsilon f\left(\vartheta_0
,t\right)\,\hat Y_t\left(\vartheta _0\right)$ containing unknown parameter.
The increase in $\psi _\varepsilon $ leads  to increase in the
signal-to-noise ratio, $SNR \approx \frac{\psi _\varepsilon ^2}{\varepsilon
  ^2}$.
Of course, all these explications are only heuristics and for the other range of $\psi
_\varepsilon $ they may not apply.  In particular, if $\psi _\varepsilon
=\varepsilon $,  the rate of convergence of the estimators is essentially
better than $\sqrt{\frac{\varepsilon }{\psi _\varepsilon }}$ (see \eqref{07}).

\bigskip

The normalization $\varphi _\varepsilon $ is determined by the key
representation \eqref{2-14}. We obtained $\varphi
_\varepsilon=\sqrt{\frac{\varepsilon }{\psi _\varepsilon }}=\varepsilon
^{\frac{1}{2}\left(1-\delta \right)}=\varepsilon ^\gamma , \frac{1}{3}<\gamma
<\frac{1}{2}$ because we assumed that $\varepsilon /\psi _\varepsilon
^3\rightarrow 0$ or let $\psi _\varepsilon =\varepsilon ^\delta ,0<\delta
<\frac{1}{3}$. If we assume that $\varepsilon /\psi _\varepsilon
^3\rightarrow \infty $ then the main term in \eqref{2-14} will have order
$O\left(\frac{\varepsilon }{\psi _\varepsilon }\right)$ and this would probably
lead to the normalization $\varphi _\varepsilon =\psi _\varepsilon
=\varepsilon ^\delta ,\frac{1}{3}<\delta <1$.

\bigskip

{\sl Remark.} The case $y_0\not=0$ merits a special study because there is an ``atom'' at
the point $t=0$. Note that the ordinary integral
in the expression \eqref{ku} converges to the similar limit
($t=\frac{\varepsilon }{\psi _\varepsilon } \ln\varepsilon ^{-1}\rightarrow 0$) 
\begin{align*}
\frac{\psi _\varepsilon}{\varepsilon }\int_{0}^{t}e^{-\int_{s}^{t}q_\varepsilon \left(\vartheta_0
  ,v\right){\rm d}v} A\left(\vartheta_0 ,s\right)\dot  f\left(\vartheta_0
    ,s\right)m\left(\vartheta_0
  ,s\right){\rm d}s\longrightarrow \frac{\dot f\left(\vartheta
  _0,0\right)}{f\left(\vartheta _0,0\right)}y_0 .
\end{align*}
To verify it we use the expansion $\gamma _*\left(\vartheta
_0,t\right)=\frac{\psi _\varepsilon }{\varepsilon }b\left(\vartheta
_0,0\right)^2t\left(1+o\left(1\right)\right) $  for small $t$ and change the variables $v=z
\frac{\varepsilon }{\psi _\varepsilon }$, $s=y\frac{\varepsilon }{\psi
  _\varepsilon }$. The other two integrals in  \eqref{ku} converge to
zero. Therefore, once more we have
\begin{align*}
\dot f\left(\vartheta _0,t\right)m\left(\vartheta _0,t\right)+f\left(\vartheta
_0,t\right)\dot m\left(\vartheta _0,t\right)\longrightarrow 0.
\end{align*}
The rate of convergence of estimators  has to be at least $\sqrt{\frac{\varepsilon}{ \psi
  _\varepsilon }}$. For example, if we suppose that $f\left(\vartheta
,t\right)=\vartheta f\left(t\right), f\left(0\right)\not=0$ and introduce the estimator
\begin{align*}
\bar\vartheta _{\tau _\varepsilon }=\frac{X_{\tau _\varepsilon }}{y_0 f\left(0\right)\tau _\varepsilon },
\end{align*}
then for  $\tau _\varepsilon =\varepsilon
/\psi _\varepsilon $ using elementary calculations  we obtain 
\begin{align*}
\sqrt{\frac{ \psi
  _\varepsilon}{\varepsilon} }\left(\bar\vartheta _{\tau _\varepsilon }-\vartheta _0\right)
 \Longrightarrow {\cal N}\left(0, d\left(\vartheta _0\right)^2\right)
\end{align*}
with some $d\left(\vartheta _0\right)^2>0$. Of course, this estimator can be
used as preliminary in the construction \eqref{3-3} of one-step MLE-process
$\vartheta _{t,\varepsilon }^\star,\tau _\varepsilon <t\leq T$ for this
model. Note that such ($\tau _\varepsilon \rightarrow 0$) preliminary
estimators in one-step estimation were used many times (see, e.g. \cite{KZ14},
\cite{Kut20a} and references there in), but the rate of convergence of them
were always slower than the rate of convergence of estimators constructed by
observations on a fixed interval.

{\bf Acknowledgment.} I am grateful to P. Chigansky for useful comments and
especially for attracting my attention to the situation with $y_0\not=0$ (see
{\sl Remark} above).
This research was supported by RSF project no 20-61-47043.

\end{document}